\documentclass[preprint,12pt]{elsarticle}

\usepackage[margin=1in]{geometry}  
\usepackage{graphicx}              
\usepackage{amsmath}               
\usepackage{amsfonts}              
\usepackage{amssymb}
\usepackage{amsthm}                
\usepackage{mathtools}		
\usepackage{enumitem}		
\usepackage{hyperref}

\usepackage{xcolor}  
\usepackage{lmodern}

\newtheorem{thm}{Theorem}[section]
\newtheorem{lem}[thm]{Lemma}

\newtheorem{cor}[thm]{Corollary}

\theoremstyle{definition}
\newtheorem{defi}[thm]{Definition}
\newtheorem{exa}[thm]{Example}
\newtheorem{rem}[thm]{Remark}

\DeclareMathOperator{\id}{id}

\newcommand{\cdu}{\mathcal{C}_d^{\Pi_{d-1}}}

\newcommand{\cd}{\mathcal{C}_d}
\newcommand{\qd}{\mathcal{Q}_d}
\newcommand{\md}{\mathcal{M}_d}
\newcommand{\cds}{\hat{\mathcal{C}}_d}
\newcommand{\qds}{\hat{\mathcal{Q}}_d}
\newcommand{\xb}{\mathbf{x}}
\newcommand{\cdus}{\hat{\mathcal{C}}_d^{\Pi_{d-1}}}
\newcommand{\cdb}{\mathcal{C}_2}
\newcommand{\qdb}{\mathcal{Q}_2}
\newcommand{\cdbs}{\hat{\mathcal{C}}_2}

\providecommand{\FD}[1]{{\color{black}#1}}   

\begin{document}


\title{On families of copulas and quasi-copulas homeomorphic to the Hilbert cube
and related results}

\author[FD]{Fabrizio Durante}
\ead{fabrizio.durante@unisalento.it}
\author[JFS]{Juan Fern\'andez S\'anchez}
\ead{jfsjufesa@gmail.com}
\author[WT]{Wolfgang Trutschnig\corref{cor1}}
\ead{wolfgang@trutschnig.net}

\address[FD]{Dipartimento di Matematica e Fisica ``Ennio De Giorgi'', Università del Salento, 73100 Lecce, Italy}
\address[JFS]{Almería, Spain}
\address[WT]{Department for Artificial Intelligence and Human Interfaces, University of Salzburg, 5020 Salzburg, Austria}

\cortext[cor1]{Corresponding author}

\begin{abstract}
Working with tools from infinite-dimensional topology and using measure-theoretic ideas we analyze 
numerous well-known fami\-lies of copu\-las and quasi-copulas from a topological perspective, prove that most of them
are homeomorphic to the Hilbert cube $\mathcal{H}$ and hence have the Fixed Point as well as the Disjoint Cells property. Working with so-called $Z$-sets, we 
then study how large one set is within another.  
Doing so, in order to show that a set $A$ is small/topologically negligible 
within a compact metric space $(\Omega,m)$, we construct a homeomorphism 
$h: [0,1] \times \mathcal{H} \to \Omega$ mapping $\{0\} \times \mathcal{H}$ to $A$. 
The obtained results are to a certain extent surprising since various 
supposedly small families are topologically large.  \\ \\
\medskip
\noindent MSC 2020: 62H05, 54F65.
\end{abstract}

\begin{keyword}
Copula \sep Quasi-copula \sep Hilbert cube \sep topological size
\end{keyword}
\maketitle

\section{Introduction}
\FD{Copulas and quasi-copulas are by now standard tools in probability, statistics and aggregation theory (see \cite{DuSe,Nel}). Over the last decades a large number of subclasses of the family $\cd$ of $d$-dimensional copulas and of the family $\qd$ of $d$-dimensional quasi-copulas have been introduced and studied, each one characterised by some analytic, algebraic or probabilistic property.

A natural question is how ``large'' such a subclass is within the family it is naturally embedded in. Classically this question has been addressed in terms of Baire category: see, for instance, \cite{DurFerTru15JMAA,DurFerTru16FSS,DurFerTru16Demo,DuranteFernandezTrutschnigUbeda}.

Recently, a finer topological point of view has been introduced by working with tools from infinite-dimensional topology. Liu and Yang \cite{LiuYang2023} showed that the space $(\cd,d_\infty)$ is homeomorphic to the Hilbert cube $\mathcal{H}$; analogous results were obtained for Schur-concave copulas \cite{Liu2024a}, for the spaces of diagonals and opposite diagonals \cite{LiuLiuJiang2024}, and for various families of bivariate copulas in \cite{FST.l2}. By Keller's theorem, however, every infinite-dimensional compact convex subset of a normed space is homeomorphic to $\mathcal{H}$, so the mere fact that a family is homeomorphic to $\mathcal{H}$ does not say anything about its size relative to another one. What distinguishes a ``small'' subfamily $A$ from the ambient family $\Omega$ is the way in which $A$ sits inside $\Omega$. In this paper we measure this by means of $Z$-sets: a closed set $A\subseteq\Omega$ is a $Z$-set if $\Omega$ can be mapped into $\Omega\setminus A$ by continuous maps arbitrarily close to the identity. Every $Z$-set is nowhere dense, but not vice versa, so being a $Z$-set is a strong form of topological negligibility. Moreover, if $A\subseteq\Omega$ are both homeomorphic to $\mathcal{H}$ and $A$ is a $Z$-set of $\Omega$, the $Z$-set unknotting theorem yields a homeomorphism $h:[0,1]\times\mathcal{H}\to\Omega$ with $h(\{0\}\times\mathcal{H})=A$, i.e., the pair $(\Omega,A)$ is topologically equivalent to the pair formed by the Hilbert cube and one of its faces. Building on this observation, Theorem~\ref{thm:main.tool} provides a simple and flexible criterion for nested chains $A_1\subset\cdots\subset A_n$ of convex compact sets, which is the main tool of the paper.

In this work, we show that the families of diagonals and of convex diagonals, of copulas and quasi-copulas with given diagonal, of lower semilinear copulas, of biconic copulas, of diagonal and Bertino copulas, and of extreme-value copulas are all homeomorphic to $\mathcal{H}$ and are $Z$-sets in the relevant ambient family ($\cdb$, $\cdbs$, or $\qdb^{\delta}$). Moreover, for the class of  
weakly Schur-concave copulas, we prove that there is a homeomorphism $h:[0,1]^3\times\mathcal{H}\to\cdb$ mapping the corresponding faces onto $\cdb^{wsc}$, $\cdbs^{wsc}$ and $\cdb^{sc}$, respectively. This latter result clarifies the ambit of applicability of a previous result in \cite{Liu2024a}. Finally, in arbitrary dimension we give a short alternative proof of one of the main results in \cite{LiuYang2023}, show that $\cd$ is a $Z$-set of $\qd$, and study exchangeable copulas and quasi-copulas, copulas with independent $(d-1)$-dimensional marginals, spatially homogeneous copulas and $\alpha$-migrative copulas.

Some of these results are, at first glance, surprising. For example, the family of $d$-dimensional copulas all of whose $(d-1)$-dimensional marginals are independence copulas, or the family of extreme-value copulas (which is parametrised by a single convex function of one variable), are topologically as rich as the whole family $\cd$ (or $\cdb$), although they are negligible inside it in the sense of $Z$-sets. 

As by-products we obtain that all the families considered have the fixed point property and the disjoint-cells property, and that each of them is a retract of the ambient family.

The paper is organised as follows. Section~2 collects notation and the necessary tools from infinite-dimensional topology. Section~3 deals with bivariate families related to diagonal sections and to other univariate functions, including extreme-value and (weakly) Schur-concave copulas. Section~4 studies the full families $\cd$ and $\qd$, their exchangeable subclasses and copulas with independent $(d-1)$-dimensional marginals. Sections~5 and~6 are devoted to multivariate spatially homogeneous and $\alpha$-migrative copulas, respectively.}

\section{Notation and preliminaries}

Throughout the paper, $\mathcal{H}=[-1,1]^\mathbb{N}$ denotes the Hilbert cube, 
we will denote elements of $\mathcal{H}$ and vectors by bold symbols and write, e.g., 
$\mathbf{x}=(x_1,x_2,\ldots) \in \mathcal{H}$. Considering the metric $\rho$ on $\mathcal{H}$, given 
by $\rho(\mathbf{x},\mathbf{y})=\sum_{i=1}^\infty \frac{1}{2^i}\vert x_i-y_i \vert $ 
it is straightforward to verify that $\rho$ is a metrization of coordinate-wise convergence 
and that $(\mathcal{H},\rho)$ is a convex and compact metric space.  
The Hilbert cube $(\mathcal{H},\rho)$ is commonly viewed as `prototype' compact metric space, 
since every compact metric space $(\Omega,m)$ is homeomorphic to a compact subset 
of $(\mathcal{H},\rho)$; more generally, a metric space $(\Omega,m)$ is separable if and only if 
it can be embedded in $(\mathcal{H},\rho)$, see, e.g., \cite{vanMill1989}. 

For every metric space $(\Omega,m)$ the Borel $\sigma$-algebra on $\Omega$ 
will be denoted by $\mathcal{B}(\Omega)$ and we will write $\mathcal{P}(\Omega)$ 
for the family of all probability measures on $\mathcal{B}(\Omega)$. 
The Dirac measure at $x \in \Omega$ will be denoted by $\epsilon_x$, the 
indicator function of a set $A \subseteq \Omega$ by $\mathbf{1}_A$. 
We consistently let $d \geq 2$ denote the dimension, (random and scalar) vectors will be denoted 
in bold symbols. The $d$-dimensional Lebesgue measure on $[0,1]^d$ will be denoted by $\lambda_d$ 
(we simply write $\lambda$ for the one-dimensional Lebesgue measure). 

$\cd$ and $\qd$ will denote the family of all $d$-dimensional copulas and quasi-copulas, respectively.
$M$ denotes the minimum copula, $W \in \qd$ the Fr\'echet-Hoeffding lower bound, and $\Pi$ 
the product copula. 
As usual, $d_\infty$ denotes the uniform metric on $\cd$ and $\qd$, it is well known that 
$(\cd,d_\infty)$ and $(\qd,d_\infty)$ are convex and compact metric spaces.
Every copula $C \in \cd$ corresponds to a unique $d$-stochastic measure $\mu_C$ and vice versa, 
the family of all $d$-stochastic measures will be denoted by $\md$. 
We will call a copula $C \in \cd$ singular if its corresponding $d$-stochastic measure is singular
with respect to the Lebesgue measure $\lambda_d=\mu_\Pi$ on $[0,1]^d$. 
Using Lipschitz-continuity of copulas it is straightforward to verify that 
for arbitrary $C,C_1,C_2,\ldots \in \cd$ the following \FD{two} assertions are equivalent:
\begin{enumerate}
\item $\lim_{n \rightarrow \infty} d_\infty(C_n,C)=0$.
\item The sequence $(\mu_{C_n})_{n \in \mathbb{N}}$ converges weakly to $\mu_C$.
\end{enumerate} 
A quasi-copula $Q \in \qd$ is called exchangeable, if for every permutation $\kappa$ of 
$[d]\FD{:=}\{1,\ldots,d\}$ and every $\mathbf{x}=(x_1,\ldots,x_d) \in [0,1]^d$ we have 
$Q(x_1,x_2,\ldots,x_d)=Q(x_{\kappa(1)},x_{\kappa(2)},\ldots,x_{\kappa(d)})$.
We will let $\cds$ and $\qds$ denote the family of all exchangeable copulas and quasi-copulas, 
respectively (throughout the whole paper, every family wearing a hat only consists of 
exchangeable elements). For more background on copulas and quasi-copulas
we refer to \cite{DuSe,Nel}.

Obviously, $\cd$ and $\qd$ are subsets of the Banach space $(\FD{C([0,1]^d)}, \Vert \cdot \Vert_\infty)$
of all real-valued continuous functions on the unit cube $[0,1]^d$. 
Considering that $\cd$ and $\qd$ are convex but obviously not closed w.r.t. linear combinations, we 
will refer to a set $\{Q_1,\ldots,Q_n\}$ of pairwise different quasi-copulas as linearly independent, if 
$\{Q_1,\ldots,Q_n\}$ are linearly independent as elements in $(C([0,1]^d), \Vert \cdot \Vert_\infty)$, 
\FD{i.e.}, if $\sum_{i=1}^n a_i Q_i =\mathbf{0} \in C([0,1]^d)$ for 
some $a_1,\ldots,a_n \in \mathbb{R}$ implies
$a_1=a_2=\ldots=a_n=0$. 


For every $C \in \cd$ and every vector
$\mathbf{j} = (j_1, \ldots, j_l)$ of $l \le d-1$ distinct indices in $[d]$,
we will let $C^{\mathbf{j}}$ denote the marginal copula corresponding to
$\mathbf{j}$, i.e., the copula corresponding to the push-forward
$\mu_{\FD{C^{\mathbf{j}}}} = \mu_C^{\pi_{\mathbf{j}}}$ of $\mu_C$ via the projection 
$\pi_{\mathbf{j}}: [0,1]^d \to [0,1]^l$, given by 
$\pi_{\mathbf{j}}((x_1,\ldots,x_d))=(x_{j_1},\ldots,x_{j_l})$; 
$C^{1:l}$ denotes the marginal copula of the first $l$ coordinates - for $l=1$ obviously
$C^{1:l}$ is no copula, since we have $\mu_C^{\pi_1}=\lambda$, however, we will stick to 
the notation $C^{1:l}$ also in this case. $C^{-j_0}$ will denote 
the marginal copula of all coordinates except coordinate $j_0$.

We call a map 
$K: [0,1]^{m}\times\mathcal{B}([0,1]^l) \rightarrow [0,1]$ with $m,l\in \FD{[d-1]}$ and $m+l=d$, 
a \emph{Markov kernel} from $[0,1]^{m}$ to $[0,1]^l$, if the function 
$\mathbf{x}\mapsto K(\mathbf{x},E)$ is 
$\mathcal{B}([0,1]^{m})$-$\FD{\mathcal{B}([0,1])}$-measurable for every fixed 
$E\in\mathcal{B}([0,1]^l)$ and the map $E\mapsto K(\mathbf{x},E)$ is a probability measure 
on $\mathcal{B}([0,1]^l)$ for every $\mathbf{x}\in[0,1]^{m}$. 
Given \FD{$[0,1]^m$- and $[0,1]^l$-valued random vectors} $\mathbf{X}$ and $\mathbf{Y}$, respectively, 
on a joint probability space $(\Omega,\mathcal{A},\mathbb{P})$, a Markov kernel 
$K$ is called a regular conditional distribution of
$\mathbf{Y}$ given $\mathbf{X}$, if (and only if) for every set 
$E \in \mathcal{B}([0,1]^l)$ the identity 
$$
K(\mathbf{X}(\omega), E) = \mathbb{E}(\mathbf{1}_E \circ \FD{\mathbf{Y}} | \mathbf{X})(\omega)
$$
holds for $\mathbb{P}$-almost every $\omega \in \Omega$.
It is a well-known fact (see \cite{Kallenberg}) that for each pair $(\mathbf{X}, \mathbf{Y})$ 
such a regular conditional distribution $K$ of $\mathbf{Y}$ given $\mathbf{X}$ exists and that it is 
unique for $\mathbb{P}^{\mathbf{X}}$-almost every $\mathbf{x} \in [0,1]^m$.
For $(\mathbf{X}, \mathbf{Y}) \sim C$ (i.e., if $(\mathbf{X}, \mathbf{Y})$ has copula 
$C$ as distribution function restricted to $[0,1]^d$) we will let 
$K_{C}:[0,1]^{m} \times \mathcal{B}([0,1]^l) \to [0,1]$ denote (a version of) the corresponding 
conditional distribution of $\mathbf{Y}$ given $\mathbf{X}$; $K_C$ will simply be referred to as 
(a version of) the $m$-to-$l$ Markov kernel of the copula $C$. \\
For every $G \FD{\in \mathcal{B}([0,1]^d)}$ and $\mathbf{x} \in [0,1]^{m}$ 
define the $\mathbf{x}$-section $G_\mathbf{x}$ of $G$ by $G_{\mathbf{x}}:=\{y\in [0,1]^l: 
(\mathbf{x},\mathbf{y}) \in G\}\in\mathcal{B}([0,1]^{l})$. 
Applying \emph{disintegration} of $\mu_C$ into the marginal $\mu_{C^{1:m}}$ and the 
$m$-to-$l$-Markov kernel $K_C$ of $C$
(see \cite[Section 5]{Kallenberg}\FD{)} the following identity 
holds for all $G \in \mathcal{B}([0,1]^d)$:
\begin{align}\label{eq:DI}
	\mu_C(G) = \int_{\FD{[0,1]}^{m}} K_{C}(\mathbf{x},G_{\mathbf{x}})
	\, \mathrm{d}\mu_{C^{1:m}}(\mathbf{x}).
\end{align}
In accordance with \cite{GJT} we will call a copula $C \in \cd$ \emph{completely dependent} 
(on the first $(d-1)$ coordinates), if
there exists some $\mu_{C^{-d}}$-$\lambda$ preserving transformation $g$ (i.e., a measurable
transformation fulfilling that the push-forward $(\mu_{C^{-d}})^g$ of $\mu_{C^{-d}}$ via $g$ 
coincides with $\lambda$) such that
$K(\mathbf{x},F):=\mathbf{1}_F(g(\mathbf{x}))$ is a regular conditional distribution of $C$. 
For $(\mathbf{X},Y)\sim C$ we have that complete dependence of $C$ is equivalent to the
existence of some $\mu_{C^{-d}}$-$\lambda$ preserving transformation $g$ such that
$Y=g(\mathbf{X})$ almost surely. It is straightforward to verify (see \cite{GJT}) 
that $C \in \cd$ is completely dependent if and only if there exists some 
$\mu_{C^{-d}}$-$\lambda$-preserving transformation $g: [0,1]^{d-1} \to [0,1]$ 
such that the graph 
$$
\Gamma(g)=\{(x_1,\ldots,x_{d-1},g((x_1,\ldots,x_{d-1}))):\, x_1,\ldots,x_{d-1} \in [0,1] \}
$$ 
has full $C$-measure, i.e., if $\mu_C(\Gamma(g))=1$ holds. \\
Note that in the case $m=1,l=d-1$ we have $\mu_{C^{1:m}}=\lambda$, so the integral in 
\eqref{eq:DI} boils down to an integral with respect to $\lambda$. 
Working with $1$-$(d-1)$ kernels allows to define the metric $D_1$ on $\cd$ 
by (for $\mathbf{y}=(y_1,\ldots,y_{d-1}) \in [0,1]^{d-1}$ we write 
$[\mathbf{0},\mathbf{y}] = \times_{i=1}^{d-1}[0,y_i]$)
\begin{equation}\label{eq:D1multi}
D_1(C_1,C_2) = \int_{[0,1]}\int_{[0,1]^{d-1}} \vert K_{C_1}(x,[\mathbf{0},\mathbf{y}]) 
- K_{C_2}(x,[\mathbf{0},\mathbf{y}])\vert \, d\lambda_{d-1}(\mathbf{y})\, d\lambda(x).
\end{equation}
According to \cite{D1multi}, the metric space $(\cd,D_1)$ is separable and complete and for 
copulas $C,C_1,C_2,\ldots$ with $1$-to-$(d-1)$ kernels $K_{C},K_{C_1},K_{C_2},\ldots$ we have
that weak convergence of the probability measures $(K_{C_n}(x,\cdot))_{n \in \mathbb{N}}$ 
to $K_{C}(x,\cdot)$ for $\lambda$-almost every $x \in [0,1]$ implies $D_1$-convergence of 
$(C_n)_{n \in \mathbb{N}}$ to $C$, which in turn implies 
$\lim_{n \to \infty} d_\infty(C_n,C)\FD{=0}$.

In the sequel we will study the interrelation between numerous closed subclasses of $\cd$ and 
$\qd$, and show that many of them are homeomorphic to the Hilbert cube $\mathcal{H}$. 
For the sake of readability, the families are defined in the respective sections.  
One construction we will frequently use is that of an ordinal sum of copulas/quasi-copulas.  
For the general construction we refer to \cite{DuSe,FNU,MS}, here we only recall the ordinal sum of 
two copulas/quasi-copulas: Considering $Q_1,Q_2 \in \qd$ and an arbitrary $a \in \FD{(0,1)}$, define
the ordinal sum $\mathcal{OS}_a(Q_1,Q_2)$ of $Q_1$ and $Q_2$ with respect to the partition 
$[0,a],(a,1]$ induced by $a$ as follows: 

\begin{align}\label{eq:os}
\mathcal{OS}_a(Q_1,Q_2)(\mathbf{u}) = 
\begin{cases}
a\, Q_1\!\left(\tfrac{\min(u_1,a)}{a}, \ldots, \tfrac{\min(u_d,a)}{a}\right), & M(\mathbf{u}) \in [0,a], \\[3mm]
a + (1-a)\, Q_2\!\left(\tfrac{u_1 - a}{1-a}, \ldots, \tfrac{u_d - a}{1-a}\right), & M(\mathbf{u}) \in (a,1]
\end{cases}
\end{align}


It is well-known (and straightforward to verify, see \cite{FNU}) that for $Q_1,Q_2 \in \qd$ also the ordinal sum 
fulfills 
$\mathcal{OS}_a(Q_1,Q_2) \in \qd$ and that for $C_1,C_2 \in \cd$ we have
$\mathcal{OS}_a(C_1,C_2) \in \cd$. Moreover, if either $Q_1$ or $Q_2$ is a \FD{proper} quasi-copula
and $a \in (0,1)$, then $\mathcal{OS}_a(Q_1,Q_2) \in \qd \setminus \cd$, i.e., 
$\mathcal{OS}_a(Q_1,Q_2)$ is a \FD{proper} quasi-copula \FD{(i.e., an element of $\qd\setminus\cd$)}. 

Analogous to \cite{Liu2024a,LiuYang2023,LiuLiuJiang2024} we will work with tools from 
infinite-dimensional topology to establish our main results. 
For the sake of completeness we here recall the main notions
and some results used in the sequel, for more background we refer to the textbooks \cite{vanMill1989} 
and \cite{Sakai2020}.   

\begin{defi}
Let $(\Omega,m)$ be a metric space and $A$ a closed subset of $\Omega$. Then $A$ is called a
\emph{retract} \FD{of $\Omega$} if there exists a continuous mapping $r: \Omega \rightarrow A$ that has 
each $x \in A$ as fixed point. Such a mapping $r$ is called a \emph{retraction}. \\
\FD{Moreover}, $(\Omega,m)$ is an \emph{absolute retract} (AR, for short), if it is a
retract of every metric space containing $\Omega$ as a closed subset.
\end{defi}
\noindent It is well known (and not hard to verify) 
that every retract of an AR is an AR (see, e.g., \cite{Hu1965})\FD{.}

\begin{defi}
A subset $A$ of $(\Omega,m)$ is a called a \emph{$Z$-set} if it is closed and for every
continuous mapping $\varepsilon: \Omega \rightarrow (0,1)$ there exists some continuous function 
$f: \Omega \rightarrow \Omega\setminus A$ such that $m(x,f(x)) < \varepsilon(x)$ holds for all
$x \in \Omega$. 
\end{defi}
In the case $(\Omega,m)$ is compact, it obviously suffices to consider $\varepsilon$ constant.
We interpret $Z$-sets as topologically negligible (and hence topologically small) sets. 
It is straightforward to verify that every $Z$-set in $(\Omega,m)$ is nowhere dense (and hence of 
first Baire category) in $(\Omega,m)$. 
\FD{Moreover}, every singleton in $\mathcal{H}$ is a $Z$-set and 
that for every fixed $i \in \mathbb{{N}}$
so is the set $\{\xb \in \mathcal{H}: x_i \in \{-1,1\}\}$.  

Homeomorphisms between $Z$-sets in $\mathcal{H}$ are extendable - the following result holds (see \cite{Sakai2020,vanMill1989} and the original references therein):

\begin{thm}[$Z$-Set Unknotting]\label{thm:unknot}
Let $A$ and $B$ be $Z$-sets of $\mathcal{H}$ and $h : A \to B$ be a
homeomorphism. Then $h$ is extendable to $\mathcal{H}$, i.e., there exists 
some homeomorphism $h^*: \mathcal{H} \to \mathcal{H}$ which coincides with $h$ on $A$. 
\end{thm}

\begin{defi}
We say that a metric space $(\Omega,m)$ has the \emph{disjoint-cells property} (DC property, for short) 
if for every $n \in \mathbb{N}$, every continuous
function $f: [0,1]^n \times \{0,1\} \to \Omega$, and every $\varepsilon > 0$, there exists a 
continuous function $g: [0,1]^n \times \{0,1\} \to \Omega$, such that
$$
\sup\{m(f(\xb),g(\xb)):\, \xb \in [0,1]^n \times \{0,1\} \} < \varepsilon \quad \text{and} \quad 
g([0,1]^n \times \{0\}) \cap g([0,1]^n \times \{1\}) = \emptyset.
$$
\end{defi}

The following simple lemma states a sufficient condition for the DC property to hold for compact 
metric spaces:
\begin{lem}
Suppose that $(\Omega,m)$ is a compact metric space. If for every $\varepsilon > 0$ there
exist continuous functions $f,g : \Omega \to \Omega$ such that $f(\Omega)$ 
and $g(\Omega)$ are
disjoint and $m(f(x),x) < \varepsilon$, $m(g(x),x) < \varepsilon$ for all
$x \in \Omega$, then $(\Omega,m)$ has the DC property.
\end{lem}

The last `ingredient' we need is the so-called small inductive dimension, denoted by `ind'
and defined recursively. 
\begin{defi}[Small inductive dimension]
Let $(\Omega,m)$ be an arbitrary metric space. Then the small inductive dimension 
$\operatorname{ind}(\Omega) \in \{-1,0,1,2,\ldots\} \cup \{\infty\}$ of $\Omega$ 
is defined as follows:  
$\operatorname{ind}\emptyset = -1$. 
$\Omega\FD{\neq\emptyset}$ has small inductive dimension $0$ if there exists a base of open sets consisting of
clopen sets; in this case $\operatorname{ind}\Omega = 0$. For a positive integer
$k$, $\operatorname{ind}\Omega \le k$ if there exists a base $\mathcal{U}$ of open sets of
$\Omega$ with $\operatorname{ind}\partial U \le k-1$ for every $U \in \mathcal{U}$. 
We say $\operatorname{ind}\Omega = k$ if $\operatorname{ind}\Omega \le k$ and
\FD{$\operatorname{ind}\Omega \le k-1$ does not hold}. If \FD{there is no $k\in\{-1,0,1,\ldots\}$ with $\operatorname{ind}\Omega\le k$}, we set $\operatorname{ind}\Omega := \infty$.
\end{defi}
It is well known that the small inductive dimension is monotonic and invariant under homeomorphisms\FD{;} the following lemma holds (see \cite{Edgar2008,HurewiczWallman1941}):
\begin{lem}\label{lem:ind.mono}
If $(\Omega_1,m_1)$ and $(\Omega_2,m_2)$ are homeomorphic, then $\operatorname{ind}\Omega_1 = \operatorname{ind}\Omega_2$. If $\Omega_1 \subseteq \Omega_2$, then 
$\operatorname{ind}\Omega_1 \leq \operatorname{ind}\Omega_2$. 
Moreover, for every \FD{$n\in\mathbb{N}$} we have \FD{$\operatorname{ind}\mathbb{R}^n = n$}. 
\end{lem}
\begin{defi}[Fixed Point Property]
We say that a metric space $(\Omega,m)$ has the Fixed Point property (FP property, for short) if every continuous mapping $f: \Omega \to \Omega$
has at least one fixed point.
\end{defi}
It is well-known that $(\mathcal{H},\rho)$ has the FP property and so has
every metric space $(\Omega,m)$ being homeomorphic to $(\mathcal{H},\rho)$.

The following two theorems are fundamental and underline the importance of the Hilbert cube 
(see \cite{Sakai2020} and the references therein).
\begin{thm}[Toru\'nczyk characterization]\label{thm:torun}
A metric space $(\Omega,m)$ is homeomorphic to $\mathcal{H}$ if and only if it is compact,
an \FD{AR}, and has the DC property.
\end{thm} 
\begin{thm}[Keller--Dobrowolski--Toru\'nczyk]\label{thm:keller}
Let $A$ be a convex subset of a normed space $(N,\Vert \cdot \Vert)$. Then $A$ is homeomorphic to
$\mathcal{H}$ if and only if $A$ is compact and $\operatorname{ind}A=\infty$. 
\end{thm}

The following theorem will be our main tool for establishing that various subsets of $\cd$ and $\qd$ 
are \FD{homeomorphic} to $\mathcal{H}$. 
\begin{thm}\label{thm:main.tool}
Suppose that $A_1 \subset A_2 \subset \cdots \subset A_n$ are convex, compact subsets
of a normed space $(N, \Vert \cdot \Vert)$, that $\operatorname{ind}A_1 = \infty$, and 
that $A_{i-1}$ is a $Z$-set of $A_i$ for every $i \in \{2,\ldots,n\}$. 
Then there exists a homeomorphism $h : [0,1]^{n-1} \times \mathcal{H} \to A_n$ fulfilling 
\begin{align*}
h(\{0\}\times[0,1]^{n-2}\times\mathcal{H}) &= A_{n-1}, \quad
h(\{0\}\times\{0\}\times[0,1]^{n-3}\times\mathcal{H}) = A_{n-2}, \ldots \\
h(\{0\}\times\{0\}\times\cdots\times\{0\}\times\mathcal{H}) &= A_1 .
\end{align*}
Moreover, every $A_i$ has the DC and the FP property, and $A_j$ is a retract of $A_i$ whenever $j<i$. 
\end{thm}
\begin{proof}
First of all notice that Theorem \ref{thm:keller} implies that $A_1$ is homeomorphic to 
the Hilbert cube $\mathcal{H}$ and so is every $A_i$ in the theorem. 
The existence of the stated homeomorphism can be proved by using the afore-mentioned results 
(in particular Theorem \ref{thm:unknot}, 
Lemma \ref{lem:ind.mono}, Theorem \ref{thm:keller}) in combination with the fact that
homeomorphisms preserve $Z$-sets, proceeding as in the proof of the Main Theorem 
in \cite{LiuYang2023} and using induction on $n$. 
Considering the homeomorphism $h$, the fact that each $A_i$ has the DC \FD{and the FP} property is clear \FD{(each $A_i$ is homeomorphic to $\mathcal{H}$)}. 
Moreover, using the fact that $\{0\}\times\{0\}\times\cdots\times\{0\}\times\mathcal{H}$ 
is a retract of $\{0\}\times\{0\}\times\cdots\times\{0\} \times [0,1]\times\mathcal{H}$, which 
is a retract of $\{0\}\times\{0\}\times\cdots\times\{0\} \times [0,1]^2\times\mathcal{H}$ and so on, 
shows that $A_j$ is a retract of $A_i$ whenever $j<i$. 
\end{proof}
\begin{rem}\label{rem:alternative.main}
Theorem \ref{thm:main.tool} also holds without convexity \FD{(and without an ambient normed space) if $A_1 \subset A_2 \subset \cdots \subset A_n$ are compact subsets of a metric space, each $A_i$ is homeomorphic to the Hilbert cube, and $A_{i-1}$ is a $Z$-set of $A_i$ for every $i$.}
\end{rem}

The following simple lemma on convex hulls of finite linearly independent sets will be used for deriving properties of various subsets of $\cd$ and $\qd$: 
\begin{lem}\label{lem:dimension}
Suppose that $Q_1,\ldots,Q_{n+1}$ are $n+1$ linearly independent elements of \FD{a normed space $(N,\Vert\cdot\Vert)$}
and set $\mathcal{G}=\{Q_1,\ldots,Q_{n+1}\}$.
Then the convex hull $\operatorname{convh}(\mathcal{G})$ of $\mathcal{G}$ has 
small inductive dimension equal to $n$.
Moreover, if $\mathcal{A}$ is a convex subset of \FD{$N$} containing infinitely many
linearly independent elements, then $\operatorname{ind}\mathcal{A} = \infty$. 
\end{lem}
\begin{proof}
Consider the unit simplex 
$S = \{\mathbf{p}=(p_1,\ldots,p_{n+1})\in [0,1]^{n+1}: \sum_{i=1}^{n+1} p_i = 1\}
\subset \mathbb{R}^{n+1}$ and define the mapping $f: S \to \mathcal{Q}_d$ by 
$f(\mathbf{p}) = \sum_i p_i Q_i$. Since $Q_1,\ldots,Q_{n+1}$ are linearly independent, 
$f: S \to f(S)$ is \FD{a continuous bijection from a compact onto a Hausdorff space and hence} a homeomorphism and $f(S)$ is the convex hull of $\mathcal{G}$. 
\FD{Moreover}, $S$ is homeomorphic to a subset of $\mathbb{R}^n$, so  
$\operatorname{ind}S \le n$ holds.
On the other hand, the open unit simplex 
$\{(p_1,\ldots,p_{n+1})\in \FD{(0,1)}^{n+1}: \sum_{i=1}^{n+1} p_i = 1\}$
is homeomorphic to $\mathbb{R}^n$ (see, e.g., \cite{Rockafellar1970}), implying  
$n \le \operatorname{ind}S$. Altogether we obtain $n = \operatorname{ind}S = \operatorname{ind}f(S)$, 
which completes the proof of the first assertion. The second assertion follows from the first one 
in combination with Lemma \ref{lem:ind.mono}.
\end{proof}

\section{Bivariate copulas generated by one-dimensional functions}

The families we consider in this section all have in common that they are \FD{expressed in terms of a one-dimensional function. In particular, we will consider copulas that are expressed in terms of their} diagonals. Recall that a function $\delta: [0,1] \to [0,1]$ is a diagonal of some bivariate copula, 
if and only if it fulfills the following properties (see, e.g., \cite{DuSe}) 
\begin{enumerate}
    \item[(i)] $\delta(u) \le u$ for all $u \in [0,1]$ and $\delta(1) = 1$,
    \item[(ii)] $\delta$ is non-decreasing,
    \item[(iii)] $\delta$ is 2-Lipschitz, i.e., $|\delta(v) - \delta(u)| \le 2|v-u|$ holds for all $u, v \in [0,1]$.
\end{enumerate}
\FD{We let $\mathcal{D}$ denote the family of all such functions and write $\delta_M(t)=t$, $\delta_\Pi(t)=t^2$, $\delta_W(t)=\max\{2t-1,0\}$. Notice that $\mathcal{D}$ also coincides with the set of all diagonal sections of a quasi-copula.} It is straightforward to verify \FD{(for instance, using Arzel\`a--Ascoli Theorem)} that  
$\mathcal{D}$ is a convex and compact subset of the Banach space $(C[0,1],\Vert \cdot \Vert_\infty)$.
We will let $\cdb^\delta$ and $\qdb^\delta$ denote the family of all copulas and quasi-copulas 
with diagonal $\delta \in \mathcal{D}$. Moreover, $\delta_C \in \mathcal{D}$ 
denotes the diagonal of the copula $C$. 
It is straightforward to verify that both sets are convex and closed, so 
$(\cdb^\delta,d_\infty)$ and $(\qdb^\delta,d_\infty)$ are convex and compact metric spaces.  \\
Before focusing on families of copulas related with diagonals we derive some preliminary results 
on subclasses of diagonals, which will be useful in the sequel.
Doing so, we let $\mathcal{D}_{con}$ denote the family of all convex $\delta \in \mathcal{D}$.
Since uniform limits of sequences of convex functions on $[0,1]$ are convex too, it follows that 
$\mathcal{D}_{con}$ is a convex and compact subset of $(\mathcal{D},\Vert \cdot \Vert_\infty)$.

\begin{lem}\label{lem:diagonals.general}
$\mathcal{D}_{con}$ is a compact, convex subset of $(\mathcal{D},\Vert \cdot \Vert_\infty)$ 
which contains infinitely many linearly independent elements. 
Moreover, $\mathcal{D}_{con}$ is a $Z$-set of $\mathcal{D}$.
\end{lem}
\begin{proof}
For every $a \in \FD{[1,2]}$ define $\delta_a \in \mathcal{D}_{con}$ by 
\begin{equation}\label{eq:deltaa}
\delta_a(x)=x^a
\end{equation} 
and write $\mathcal{D}^\star:=\{\delta_a: a \in [1,2]\}$. Considering that the family 
$\mathcal{D}^\star$ is linearly independent in $\FD{(C([0,1]),\Vert \cdot \Vert_\infty)}$ \FD{(if $\sum_i e_i x^{a_i}\equiv0$ with $a_1<\dots<a_n$, divide by $x^{a_1}$ and let $x\downarrow0$ to get $e_1=0$, and iterate)}, the first 
assertion follows \FD{(via Lemma~\ref{lem:dimension})}. For proving the second assertion, for every $b \in (0,1)$ define the mapping 
$f_b: \mathcal{D} \to \mathcal{D}$ by
$$
(f_b(\delta))(x):=b\,\delta\left(\tfrac{x}{b} \right)\mathbf{1}_{[0,b]}(x) + 
\left(b + (1-b)\,\delta_W \left(\tfrac{x-b}{1-b} \right)\right) \mathbf{1}_{(b,1]}(x).
$$ 
\FD{Note that $f_b(\delta)$ is the diagonal of $\mathcal{OS}_b(C,W)$ for any $C$ with $\delta_C=\delta$, so $f_b(\delta)\in\mathcal{D}$, and that $f_b$ is continuous.}
Obviously for every $b\in(0,1)$ and every $\delta \in \mathcal{D}$ we have that 
$f_b(\delta)$ is not convex \FD{(it is constant equal to $b$ on $[b,\frac{1+b}{2}]$ although $f_b(\delta)(0)=0<b$)}, so $f_b$ maps $\mathcal{D}$ into $\mathcal{D}\setminus \mathcal{D}_{con}$.
Since for arbitrary $\varepsilon>0$ we can find some sufficiently large $b \in (0,1)$ with
$$
\sup_{\delta \in \mathcal{D}} \Vert f_b(\delta) - \delta \Vert_\infty < \varepsilon,
$$
it follows that $\mathcal{D}_{con}$ is a $Z$-set of $\mathcal{D}$ \FD{(in fact $\Vert f_b(\delta)-\delta\Vert_\infty\le 3(1-b)$ by the $2$-Lipschitz property)}.
\end{proof}

Applying Theorem \ref{thm:main.tool} for the case that $n=2, A_1=\mathcal{D}_{con}, A_2=\mathcal{D}$ 
and $N=C([0,1])$ we obtain the following result:

\begin{thm}\label{thm:main.diagonals}
There exists a homeomorphism $h : [0,1] \times \mathcal{H} \to \mathcal{D}$ with 
$h(\{0\} \times \mathcal{H})= \mathcal{D}_{con}$. Moreover, $\mathcal{D}_{con}$ has the DC and the FP property and
 is a retract of $\mathcal{D}$.
\end{thm}

\FD{This latter result should be compared with \cite{LiuLiuJiang2024}, where the space of diagonals is already shown to be homeomorphic to $\mathcal{H}$.}

\subsection{Copulas and quasi-copulas with given diagonal}   

Throughout this subsection we let $\delta \in \mathcal{D}\setminus\{\delta_M\}$ be arbitrary but fixed. 
\begin{lem}
$\cdb^{\delta}$ is a compact, convex subset of $(\cdb,d_\infty)$ which contains infinitely many 
linearly independent elements. Moreover, $\cdb^{\delta}$ is a $Z$-set of $\cdb$.
\end{lem}
\begin{proof}
Compactness and convexity are obvious. The fact that $\cdb^{\delta}$ contains infinitely many 
linearly independent elements is a direct consequence of \FD{\cite[Proposition 2, first proof]{Durante2011}}. Finally, \FD{we show that $\cdb^\delta$ is a $Z$-set of $\cdb$ using ordinal sums. If there is $c<1$ with $\delta(t)=t$ for all $t\in[c,1]$ set $R:=W$, otherwise set $R:=M$. For $b\in(0,1)$ put $f_b(C):=\mathcal{OS}_b(C,R)$. On $(b,1]$ the diagonal of $f_b(C)$ equals $t$ (if $R=M$) or equals $b$ on $(b,\frac{1+b}{2}]$ (if $R=W$). In the first case $\delta$ does not coincide with the identity on $(b,1]$ by the choice of $R$; in the second case $\delta(t)=t>b$ on $(\max\{b,c\},\frac{1+b}{2}]$ for $b$ close to $1$. Hence, for all $b$ sufficiently close to $1$, $f_b$ maps $\cdb$ into $\cdb\setminus\cdb^\delta$, is continuous, and fulfills $\sup_{C\in\cdb}d_\infty(f_b(C),C)\le 3(1-b)$ (Lipschitz property of copulas). Therefore $\cdb^{\delta}$ is a $Z$-set of $\cdb$.}
\end{proof}
Applying Theorem \ref{thm:main.tool} for the case that $n=2, A_1=\cdb^{\delta}, A_2=\cdb$ 
and $N=C([0,1]^2)$ we obtain the following result:
\begin{thm}\label{thm:main.biv.fixed.diag}
Let $\delta \in \mathcal{D}\setminus \{\delta_M\}$ be arbitrary but fixed. 
Then \FD{there} exists a homeomorphism $h : [0,1] \times \mathcal{H} \to \cdb$ with 
$h(\{0\} \times \mathcal{H})= \cdb^{\delta}$. Moreover, $\cdb^{\delta}$ has the 
DC and the FP property and
 is a retract of $\FD{\cdb}$.
\end{thm}
Next we focus on the interrelation between $\cdb^{\delta}$ and  $\qdb^{\delta}$ and start with the 
following lemma.

\FD{\begin{lem}\label{lem:aux.copula}
Let $\delta\in\mathcal{D}\setminus\{\delta_M\}$. Then there exist $P\in\cdb^\delta$, points 
$0<x_1<y_1<1$ and some $r_0>0$ such that, for every $r\in(0,r_0]$, setting $x_2:=x_1+r$, 
$y_2:=y_1+r$, the square $R_r:=[x_1,x_2]\times[y_1,y_2]$ is contained in 
$\{(x,y)\in(0,1)^2:\ x<y\}$ and, for all $x\in(x_1,x_2]$ and $y\in(y_1,y_2]$, the rectangles
\begin{equation}\label{eq:aux_lemma_3.5}
[0,x_1]\times(y_1,y],\quad (x_1,x]\times[0,y_1],\quad (x_1,x_2]\times(y_2,1],\quad (x_2,1]\times(y_1,y_2]
\end{equation}
have positive $\mu_P$-measure.
\end{lem}
\begin{proof}
In fact, fix $t_0\in(0,1)$ with $\delta(t_0)<t_0$ and some $C\in\cdb^\delta$, set 
$R:=[0,t_0]\times(t_0,1]$, $m:=\mu_C(R)=t_0-\delta(t_0)>0$, and let $\nu_1,\nu_2$ be the 
marginals of the restriction of $\mu_C$ to $R$. Replacing this restriction by 
$\frac1m\nu_1\otimes\nu_2$ changes neither the marginals nor the mass of the squares $[0,t]^2$ 
(the set $[0,t]^2\cap R$ is empty for $t\le t_0$ and equals $[0,t_0]\times(t_0,t]$, with mass 
$\nu_2((t_0,t])$ under both measures, for $t>t_0$). Hence the resulting doubly stochastic 
measure is $\mu_P$ for some $P\in\cdb^\delta$. Let $x_1$ (resp.\ $y_1$) be the largest point with 
$\nu_1([0,x_1])=\frac m2$ (resp.\ $\nu_2((t_0,y_1])=\frac m2$), so that $0<x_1<t_0<y_1<1$, and 
choose $r_0>0$ with $x_1+r_0<t_0$, $\nu_1([0,x_1+r_0])<m$ and $\nu_2((t_0,y_1+r_0])<m$. 
Then $R_r\subseteq(0,t_0)\times(t_0,1)$, and since $\mu_P\ge\frac1m\nu_1\otimes\nu_2$ on $R$, 
each rectangle in \eqref{eq:aux_lemma_3.5} contains a subrectangle of $R$ of positive 
$\nu_1\otimes\nu_2$-measure. 
\end{proof}}

\FD{
\begin{lem}\label{lem:Zset.diag.quasi}
Let $\delta \in \mathcal{D}\setminus\{\delta_M\}$. Then $\cdb^{\delta}$ is a $Z$-set of $\qdb^{\delta}$.
\end{lem}
\begin{proof}
Since $\qdb^\delta$ is compact, it suffices to construct, for all sufficiently small $r>0$, 
a continuous map $f_r:\qdb^\delta\to\qdb^\delta\setminus\cdb^\delta$ with 
$d_\infty(f_r(Q),Q)\le 3r$ for every $Q\in\qdb^\delta$. 
Let $P$, $x_1,y_1$, $r_0$ and $R_r=[x_1,x_2]\times[y_1,y_2]$ be as in Lemma~\ref{lem:aux.copula} and fix $r\in(0,r_0]$.

For $Q\in\qdb^\delta$ let $Q_r:=(1-r)Q+rP\in\qdb^\delta$ and let $f_r(Q)$ 
coincide with $Q_r$ outside $R_r$ and with
$$
E(x,y):=\max\big\{Q_r(x_1,y),\,Q_r(x,y_1),\,Q_r(x_2,y)-(x_2-x),\,Q_r(x,y_2)-(y_2-y)\big\}
$$
on $R_r$. As a maximum of such functions, $E$ is non-decreasing and $1$-Lipschitz in each argument; 
by the same properties of $Q_r$ we have $E=Q_r$ on $\partial R_r$ and 
$Q_r(x_1,y_1)\le E\le Q_r\le Q_r(x_2,y_2)$ on $R_r$. Since $R_r$ lies in the interior of 
$[0,1]^2$ and does not meet the diagonal, $f_r(Q)$ is a quasi-copula with diagonal $\delta$. 
Moreover, $d_\infty(f_r(Q),Q)\le r+(x_2-x_1)+(y_2-y_1)=3r$, and $f_r$ is continuous since 
$d_\infty(f_r(Q),f_r(Q'))\le d_\infty(Q,Q')$.

Now, $f_r(Q)$ is not a copula. In fact, for every quasi-copula the four quantities
$Q(x_1,y)-Q(x_1,y_1)$, $Q(x,y_1)-Q(x_1,y_1)$, $(x_2-x_1)-(Q(x_2,y)-Q(x_1,y))$ and 
$(y_2-y_1)-(Q(x,y_2)-Q(x,y_1))$ are non-negative, and for $P$ they are bounded from below by 
the $\mu_P$-measures of the rectangles in \eqref{eq:aux_lemma_3.5}. Hence they are positive for $Q_r$ whenever 
$x>x_1$, $y>y_1$, and there is some $\kappa>0$ with 
$Q_r(x_2,y)-(x_2-x)\le Q_r(x_1,y)+(x-x_1)-\kappa$ and $Q_r(x,y_2)-(y_2-y)\le Q_r(x,y_1)+(y-y_1)-\kappa$ 
on $R_r$. Thus, for $\eta\in(0,\min\{\kappa,r\})$, on $S:=[x_1,x_1+\eta]\times[y_1,y_1+\eta]$ we have 
$E(x,y)=\max\{Q_r(x_1,y),Q_r(x,y_1)\}$, and with $a:=Q_r(x_1,y_1)$, $A:=Q_r(x_1,y_1+\eta)$, 
$B:=Q_r(x_1+\eta,y_1)$ we get
$$
V_{f_r(Q)}(S)=\max\{A,B\}-A-B+a=-\min\{A-a,\,B-a\}<0,
$$
where $V_{f_r(Q)}(S)$ denotes the $f_r(Q)$-volume of $S$. Hence $f_r(Q)\notin\cdb^\delta$, 
which completes the proof.
\end{proof}}


Having this, applying Theorem \ref{thm:main.tool} yields the following result:
\begin{thm}\label{thm:main.biv.fixed.diag2}
Let $\delta \in \mathcal{D}\setminus \{\delta_M\}$ be arbitrary but fixed. 
Then \FD{there} exists a homeomorphism $h : [0,1] \times \mathcal{H} \to \qdb^\delta$ with 
$h(\{0\} \times \mathcal{H})= \cdb^{\delta}$. Moreover, $\cdb^{\delta}$ has the 
DC and the FP property and  is a retract of $\qdb^\delta$.
\end{thm}

\subsection{Lower semilinear copulas}
The study of lower semilinear copulas (LSL, for short) goes back to \cite{DuranteKolesarovaMesiarSempi2008}. LSL copulas are in one-to-one correspondence with 
the following subclass $\mathcal{D}^{LSL}$ of $\mathcal{D}$: 
\[
\mathcal{D}^{LSL} := \left\{ \delta \in \mathcal{D} : \varphi_\delta \text{ non-decreasing}, \ \eta_\delta \text{ non-increasing} \right\}. \FD{\label{eq:DLSL}}
\]
Thereby the functions $\varphi_\delta, \eta_\delta: (0,1] \to [0,\infty)$ are given by
\[
\varphi_\delta(x) := \tfrac{\delta(x)}{x}, \qquad \eta_\delta(x) := \tfrac{\delta(x)}{x^2}.
\]
For every fixed $\delta \in \FD{\mathcal{D}^{LSL}}$ the corresponding LSL copula $S_\delta$ is then defined as
\begin{equation}\label{eq:LSL}
S_\delta(x,y) :=
\begin{cases}
y\,\tfrac{\delta(x)}{x} & \text{if } y \le x, \\[2mm]
x\,\tfrac{\delta(y)}{y} & \text{otherwise.}
\end{cases}
\end{equation}
We will denote the family of all LSL copulas by $\cdb^{LSL}$. 
Apart from their handy analytic form \eqref{eq:LSL}, according to \cite[Theorem 4.1]{MT} 
$\cdb^{LSL}$ is one of the very few (nonparametric) classes closed with respect to the 
so-called Markov product (a.k.a. star product) of copulas, i.e., for any pair of diagonals 
$\delta_1,\delta_2 \in \mathcal{D}^{LSL}$ the Markov product $S_{\delta_1}*S_{\delta_2}$ 
is an LSL copula too. 

\begin{lem}\label{lem:lsl}
$\cdb^{LSL}$ is a compact, convex subset of $(\cdbs,d_\infty)$ which contains infinitely many 
linearly independent elements. Moreover, $\cdb^{LSL}$ is a $Z$-set of $\cdbs$.
\end{lem}
\begin{proof}
Obviously $\mathcal{D}^{LSL}$ is a closed and convex subset of 
$(\mathcal{D},\Vert \cdot \Vert_\infty)$, moreover the mapping 
$\iota: \mathcal{D}^{LSL} \to \cdb^{LSL}$ assigning each diagonal its LSL copula 
is injective and continuous and preserves convex combinations, so 
$\cdb^{LSL}$ is a convex and compact subset of $(\cdbs,d_\infty)$. \\
We will work with the family $\mathcal{D}^\star \subset \mathcal{D}^{LSL}$ and 
show that $\{S_{\delta_a}: a \in (1,2)\}$ is linearly independent: 
Let $a_1,\ldots,a_n \in (1,2)$ be pairwise different and assume that there are constants 
$e_1,\ldots,e_n \in \mathbb{R}$ such that we have 
$\sum_{i=1}^n e_i \FD{S}_{\delta_{a_i}} = \mathbf{0} \in C([0,1]^2)$. 
Then for every $x \in [0,1]$ we have 
$$
0=\sum_{i=1}^n e_i \FD{S}_{\delta_{a_i}}(x,x) = \sum_{i=1}^n e_i \delta_{a_i}(x)=
\sum_{i=1}^n e_i x^{a_i}. 
$$  
Since the family $\mathcal{D}^\star$ is linearly independent in 
$C([0,1],\Vert \cdot \Vert_\infty)$ we conclude that $e_1=e_2=\ldots=e_n=0$, and it remains
to show that $\cdb^{LSL}$ is a $Z$-set of $\cdbs$. 
The latter, however, is easy to show since, e.g., for every $b \in (0,1)$, the 
mapping $f_b: \cdbs \to \cdbs$, defined by 
$f_b(\FD{C})=\mathcal{OS}_b(C,W)$, maps $\cdbs $ into $\cdbs \setminus \cdb^{LSL}$ \FD{(on $(b,\frac{1+b}{2}]$ the diagonal of $f_b(C)$ is constant, so $\varphi_\delta$ is strictly decreasing there)}.
For sufficiently large $b$, obviously $f_b$ is arbitrarily close to the identity on 
$\FD{\cdbs}$, so the proof is complete. 
\end{proof}

Applying Theorem \ref{thm:main.tool} once more yields the following: 
\begin{thm}\label{thm:lsl}
There exists a homeomorphism $h : [0,1] \times \mathcal{H} \to \cdbs$ with 
$h(\{0\} \times \mathcal{H})= \cdb^{LSL}$. Moreover, $\cdb^{LSL}$ has the DC and the FP property and
 is a retract of $\cdbs$.
\end{thm}

\subsection{A family of copulas with convex diagonals}

In \cite{Durante2011} it was shown that every $\delta \in \mathcal{D}_{con}$ induces 
a copula $H_\delta \in \cdbs$, given by
\begin{equation}\label{eq:gdelta}
H_\delta(x,y)=\min \left\{x,y,\delta(\tfrac{x+y}{2}) \right\}.
\end{equation}
We will let $\cdb^{c\delta}$ denote the family of all copulas $H_\delta$ with 
$\delta \in \mathcal{D}_{con}$. It is straightforward to verify that no copula $H_\delta$ can 
have full support\FD{: for $\delta=\delta_W$ we have $H_{\delta_W}=W$, and for $\delta\neq\delta_W$ convexity yields $\delta(\frac{1+x_0}{2})>x_0$ for some $x_0\in(0,1)$, so there is $(x,y)$ with $0<x$, $y<1$ and $H_\delta(x,y)=x$, i.e., $\mu_{H_\delta}([0,x]\times(y,1])=0$.}
Obviously the mapping assigning each $\delta \in \mathcal{D}_{con}$ 
the copula $H_\delta \in \cdb^{c\delta}$ is continuous
and bijective. Using compactness of $\mathcal{D}_{con}$ it therefore follows that 
$(\mathcal{D}_{con},\Vert \cdot \Vert_\infty)$ and $(\cdb^{c\delta},d_\infty)$ are homeomorphic, 
implying compactness of $(\cdb^{c\delta},d_\infty)$. 
We want to show that $(\cdb^{c\delta},d_\infty)$ is homeomorphic to the Hilbert cube 
$(\mathcal{H},\rho)$. Doing so, however, we can not proceed like in the previous sections since 
$\cdb^{c\delta}$ is not convex. 
In fact, considering $\delta:=\frac{1}{2}(\delta_M + \delta_W) \in \mathcal{D}_{con}$, using  
$\FD{H}_{\delta_W}=W,\FD{H}_{\delta_M}=M$, and setting $(x_0,y_0)=(\frac{1}{4},\frac{1}{2})$ yields
$$
\tfrac{1}{2}(M+W)(\tfrac{1}{4},\tfrac{1}{2}) = \tfrac{1}{8} \neq \tfrac{3}{16}=
M(\tfrac{1}{4},\tfrac{3}{16}) = \min \left\{\tfrac{1}{4},\tfrac{1}{2},\delta(\tfrac{3}{8}) \right\}
=H_\delta(\tfrac{1}{4},\tfrac{1}{2}).
$$
\FD{Since the diagonal of $\frac12(M+W)$ is $\delta$ and the only element of $\cdb^{c\delta}$ with diagonal $\delta$ is $H_\delta$, it follows that $\frac12(M+W)\notin\cdb^{c\delta}$.} We will therefore argue via $\mathcal{D}_{con}$ to prove the following lemma: 
\begin{lem}\label{lem:cdelta}
$\cdb^{c\delta}$ is homeomorphic to $\mathcal{H}$ and a $Z$-set of $\cdbs$.
\end{lem}
\begin{proof}
According to Theorem \ref{thm:main.diagonals} $\mathcal{D}_{con}$ is homeomorphic to 
the Hilbert cube $\mathcal{H}$.  Hence, using the fact that 
$(\mathcal{D}_{con},\Vert \cdot \Vert_\infty)$ and $(\cdb^{c\delta},d_\infty)$ are 
homeomorphic, directly yields that $(\cdb^{c\delta},d_\infty)$ is homeomorphic to 
$(\mathcal{H},\rho)$ too.\\
For showing that $\cdb^{c\delta}$ is a $Z$-set of $\cdbs$ we proceed as follows: 
for every 
$b \in (0,1)$ consider the mapping $f_b: \cdbs \to \cdbs$, given by
$f_b(C)=bC + (1-b)\Pi$. Then for every $b \in (0,1)$ and every $C \in \cdbs$ the copula 
$f_b(C)$ has full support, so it can not be an element of the family $\cdb^{c\delta}$ \FD{(as $b\to1$, $d_\infty(f_b(C),C)\le 1-b$)}, i.e., 
$f_b$ maps $\cdbs$ into $\cdbs \setminus \cdb^{c\delta} $. Proceeding as in the proofs of the 
previous lemmas yields the desired results.  
\end{proof}

Having this, applying Theorem \ref{thm:main.tool} (in the version sketched in 
Remark \ref{rem:alternative.main}) yields the following result:
\begin{thm}\label{thm:main.cdelta}
There exists a homeomorphism $h : [0,1] \times \mathcal{H} \to \cdbs$ with 
$h(\{0\} \times \mathcal{H})= \cdb^{c\delta}$. Moreover, $\cdb^{c\delta}$ has the DC and the FP property and
 is a retract of $\cdbs$.
\end{thm}

\subsection{Biconic copulas}

For every $\delta \in \mathcal{D}$ define the function $F_\delta: [0,1]^2 \to [0,1]$ by
\begin{equation}\label{eq:bic}
F_\delta(x,y)=
\begin{cases}
\delta\left(\tfrac{x}{x+1-y}\right)(x+1-y), & \text{ if } x\le y,\\[3mm]
\delta\left(\tfrac{y}{y+1-x}\right)(y+1-x), & \text{ if } x> y.
\end{cases}
\end{equation}
According to \cite{JwaidDeBaetsDeMeyer2012} $F_\delta$ is a bivariate copula if and only if
$\delta$ is convex. Letting $\cdb^{bic}$ denote the family of all biconic copulas \FD{$F_\delta$, $\delta\in\mathcal{D}_{con}$}, we
obviously have that $F_\delta \in \cdb^{\delta}$ for every $\delta \in \mathcal{D}_{con}$.  
As a direct consequence, the mapping $\iota: \mathcal{D}_{con} \to \cdb^{bic}$ is 
injective and preserves convex combinations. Since $\iota$ obviously is continuous 
it follows that $(\cdb^{bic},d_\infty)$ and $(\mathcal{D}_{con},\Vert \cdot \Vert_\infty)$
are homeomorphic compact metric spaces. 

\begin{lem}\label{lem:bic}
$\cdb^{bic}$ is a compact, convex subset of $(\cdb,d_\infty)$ containing infinitely many 
linearly independent elements. Moreover, $\cdb^{bic}$ is a $Z$-set of $\cdb$.
\end{lem}
\begin{proof}
For every $a \in (0,\frac{1}{2})$ let $\gamma_a \in \mathcal{D}$ be defined by
$$
\gamma_a(x)= \tfrac{x-a}{1-a}\, \mathbf{1}_{[a,1]}(x).
$$
It is straightforward to show that $\{\gamma_a: \, a \in (0,\frac{1}{2})\}$ is linearly independent.
In fact, suppose that for $n \in \mathbb{N}$ and $0< a_1 < a_2 < \ldots < a_n < \frac{1}{2}$
there are constants $e_1,\ldots,e_n \in \mathbb{R}$ such that 
$\sum_{i=1}^n e_i \gamma_{a_i}= \mathbf{0} \in C([0,1])$. Considering $x_1=\frac{1}{2}(a_1+a_2)$,  
we have $0=\sum_{i=1}^n e_i \gamma_{a_i}(x_1)=e_1 \gamma_{a_1}(x_1)$, so taking into account
$\gamma_{a_1}(x_1)>0$ yields $e_1=0$. Proceeding analogously, consi\-dering $x_2=\frac{1}{2}(a_2+a_3),
x_3=\frac{1}{2}(a_3+a_4),\ldots$\FD{, $x_n=\frac{1}{2}(a_n+1)$} eventually shows $e_1=e_2=\ldots=e_n=0$. \FD{Since the diagonal of $\sum_i e_iF_{\gamma_{a_i}}$ is $\sum_i e_i\gamma_{a_i}$, the copulas $F_{\gamma_{a_i}}$ are linearly independent as well.}  \\
It remains to \FD{show} that $\cdb^{bic}$ is a $Z$-set of $\cdb$, which, however, working with 
the mapping $C \mapsto \mathcal{OS}_b(C,W)$ for every $b \in (0,1)$, is straightforward 
to accomplish \FD{since the diagonal of $\mathcal{OS}_b(C,W)$ is $f_b(\delta_C)$, which is not convex (see the proof of Lemma~\ref{lem:diagonals.general})}.  
\end{proof}

Applying Theorem \ref{thm:main.tool} we obtain the following result:
\begin{thm}\label{thm:main.bic}
There exists a homeomorphism $h : [0,1] \times \mathcal{H} \to \cdb$ with 
$h(\{0\} \times \mathcal{H})= \cdb^{bic}$. Moreover, $\cdb^{bic}$ has the 
DC and the FP property and
 is a retract of $\cdb$.
\end{thm}

\subsection{Diagonal and Bertino copulas}

For every diagonal $\delta \in \mathcal{D}$, the diagonal copula $E_\delta$ and 
the Bertino copula $B_\delta$ are defined, \FD{respectively}, by 
\begin{eqnarray}
E_\delta(x,y) &:=& \min\left\{x,\, y,\, \tfrac{\delta(x)+\delta(y)}{2}\right\}, \\
B_\delta(x,y) &:=& M(x,y) - \min_{t\in[\min\{x,y\},\,\max\{x,y\}]}\big(t-\delta(t)\big).
\end{eqnarray}
We will let $\cdb^{diag}$ and $\cdb^{Ber}$ denote the family of all diagonal and Bertino copulas, 
respectively.  
It is well known (see \cite{p14} and original references therein) that 
$E_\delta$ is the maximal element in $\cdbs^{\delta}$ \FD{(the family of exchangeable copulas with diagonal $\delta$)} and $B_\delta$ the minimal element 
in the larger class $\cdb^{\delta}$. Following \cite{p14}, $E_\delta$ assigns all mass
to the graphs of (at most) two non-decreasing functions, $B_\delta$ to the graphs of (at most) three
non-decreasing functions, one being the identity. As a direct consequence, 
both $E_\delta$ and $B_\delta$ are singular for every $\delta \in \mathcal{D}$.  
Obviously both mappings $\delta \mapsto E_\delta$ and $\delta \mapsto B_\delta$ are injective
and continuous, implying that the three metric spaces $(\cdb^{diag},d_\infty),(\cdb^{Ber},d_\infty)$
and $(\mathcal{D},\Vert \cdot \Vert_\infty)$ are homeomorphic. In particular,  
$(\cdb^{diag},d_\infty),(\cdb^{Ber},d_\infty)$ are compact.   

\begin{lem}\label{lem:bertino}
$\cdb^{diag}$ and $\cdb^{Ber}$ are homeomorphic to the Hilbert cube $\mathcal{H}$ and 
both are $Z$-sets of $\cdbs$.
\end{lem}
\begin{proof}
We already know that $(\cdb^{diag},d_\infty),(\cdb^{Ber},d_\infty)$
and $(\mathcal{D},\Vert \cdot \Vert_\infty)$ are homeomorphic. 
Considering that, according to Theorem \ref{thm:main.diagonals},  
$(\mathcal{D},\Vert \cdot \Vert_\infty)$ is homeomorphic to the Hilbert cube $(\mathcal{H},\rho)$, 
therefore proves the first assertion.\\
Again working with the transformation $f_b(C)=bC + (1-b)\Pi$ as considered in the 
proof of Lemma \ref{lem:cdelta}, for every $b \in (0,1)$ and every $C \in \cdbs$ we have 
that $f_b(C) \in \cdbs$ \FD{has a non-trivial absolutely continuous component}, so it can neither be a diagonal copula nor a Bertino copula \FD{(both are singular)}. 
In other words, $f_b$ maps $\cdbs$ into $\cdbs\setminus(\cdb^{diag} \cup \cdb^{Ber})$ and 
it follows as in the proofs of the previous lemmas that $\cdb^{diag}$ and $\cdb^{Ber}$ are
$Z$-sets of $\cdbs$.
\end{proof}

Again using Theorem \ref{thm:main.tool} (in the version sketched in 
Remark \ref{rem:alternative.main}) we obtain the following result:
\begin{thm}\label{thm:main.diag.bertino} There exists a homeomorphism $h : [0,1] \times \mathcal{H} \to \cdbs$ with 
$h(\{0\} \times \mathcal{H})= \cdb^{diag}$. Moreover, $\cdb^{diag}$ has the 
DC and the FP property and
is a retract of $\cdbs$. 
The same two assertions also hold for $\cdb^{diag}$ replaced by $\cdb^{Ber}$.
\end{thm}
We have already shown that the metric spaces 
$(\cdb^{diag},d_\infty), (\cdb^{Ber},d_\infty),(\cdb^{bic},d_\infty),
(\cdb^{c\delta},d_\infty)$ as well as $(\mathcal{D}, \Vert \cdot \Vert_\infty)$ and 
$(\mathcal{D}_{con}, \Vert \cdot \Vert_\infty)$ 
are homeomorphic to the Hilbert cube $\mathcal{H}$. As a direct consequence, the following 
corollary holds:
\begin{cor}
There exists a homeomorphism $h_1: \mathcal{D} \to \mathcal{D}_{con}$ such that 
the mapping $E_\delta \mapsto F_{h_1(\delta)}$ 
is a homeomorphism of $(\cdb^{diag},d_\infty)$ and $(\cdb^{bic},d_\infty)$ and 
$E_\delta \mapsto H_{h_1(\delta)}$ 
is a homeomorphism of $(\cdb^{diag},d_\infty)$ and $(\cdb^{c\delta},d_\infty)$.
\end{cor} 

\subsection{Bivariate Extreme-Value Copulas}
We quickly recall the definition and properties of Extreme Value copulas. 
A copula $C \in \FD{\cdb}$ is called Extreme Value copula (EVC) if there exists some copula $B \in \FD{\cdb}$ such that
$$
C(x,y)  = \lim_{n \rightarrow \infty} B^n(x^\frac{1}{n},y^\frac{1}{n})
$$
for all $x,y \in [0,1]$. Throughout this \FD{subsection} the space of all bivariate EVCs will be denoted by 
$\cdb^{ev}$. It is well-known (see \cite{DiTru,TSFS} and the original references 
therein) that $(\cdb^{ev},d_\infty)$ is a compact metric space and that every EVC  
corresponds to a unique Pickands dependence function, i.e., a convex function $A: [0,1] \to [0,1]$
fulfilling $\max\{1-x,x\} \leq A(x) \leq 1$ for every $x \in [0,1]$. 
Letting $\mathcal{A}$ denote the family of all Pickands dependence functions, the 
one-to-one correspondence between $\cdb^{ev}$ and $\mathcal{A}$ is established via
    \begin{equation}\label{eq:map_eq_pick_copula}
    C(x,y) = (xy)^{A\left(\frac{\log(x)}{\log(xy)}\right)}, \quad x,y \in (0,1).
    \end{equation}
Letting $C_A$ denote the (unique) EVC induced by $A \in \mathcal{A}$, 
the afore-mentioned one-to-one correspondence is easily seen to be continuous, i.e., 
the mapping $\kappa: \mathcal{A} \to \cdb^{ev}$, defined by $\kappa(A)=C_A$, is a 
\FD{homeomorphism} between $(\mathcal{A},\Vert \cdot \Vert_\infty)$ and $(\cdb^{ev},d_\infty)$. 
Following \cite{Pick}, every Pickands dependence function $A$ can \FD{further} be 
identified with a so-called spectral measure $\nu$ on the unit simplex. 
Projecting $\nu$ onto $[0,1]$, according to \cite{DiTru}, $A$ corresponds to a unique
probability measure $\vartheta \in \mathcal{P}([0,1])$ fulfilling 
$$
\int_{[0,1]} x d\vartheta(x)= \tfrac{1}{2}.
$$
We will let $\mathcal{P}_\mathcal{A}$ denote the family of all these measures and refer to them 
as Pickands dependence measures.  
The one-to-one correspondence between $\mathcal{A}$ and $\mathcal{P}_\mathcal{A}$ is established via 
the identity 
\begin{equation}\label{eq:id_pick_meas}
A(t) = 1-t + 2\int_{[0,t]}\vartheta([0,z]) \mathrm{d}\lambda(z),\quad t \in [0,1].
\end{equation}
Considering the Wasserstein metric $m_W$ on $\mathcal{P}_\mathcal{A}$, 
it is straightforward to verify that the corresponding mapping 
$\Upsilon: (\mathcal{P}_\mathcal{A},m_W) \to (\mathcal{A},\Vert \cdot \Vert_\infty)$ is a 
homeomorphism that preserves convex combinations. 

A version of the following \FD{lemma} has already been proved in \cite{FST.l2} - 
we include the (slightly simplified) proof for the sake of completeness. 
\begin{lem}\label{lem:pick}
$\mathcal{A}$ is a compact, convex subset of $(C([0,1]),\Vert \cdot \Vert_\infty)$ 
that contains infinitely many linearly independent elements. 
Moreover, $\mathcal{A}$ and $\cdb^{ev}$ are homeomorphic to the Hilbert cube $\mathcal{H}$, and
$\cdb^{ev}$ is a $Z$-set of $\cdb$. 
\end{lem}
\begin{proof}
Showing that $\mathcal{A}$ contains infinitely many linearly independent elements 
can be done as follows:  
For every $a \in (0,\frac{1}{2})$ define the measure 
$\vartheta_a:=\frac{1}{2}(\epsilon_a + \epsilon_{1-a}) \in \mathcal{P}_\mathcal{A}$.
Obviously this family fulfills that for $a\neq b$ the supports of $\vartheta_a$ and $\vartheta_b$ 
are disjoint. 
As a direct consequence, the family $\{\vartheta_a: a \in (0,\frac{1}{2})\}$ is linearly 
independent in the 
vector space $\mathcal{S}([0,1])$ of all finite signed measures on $\mathcal{B}([0,1])$. 
In fact, if for $0<a_1<a_2 < \ldots < a_n < \frac{1}{2}$ there exist some 
constants $e_1,e_2,\ldots,e_n \in \mathbb{R}$ such that 
$$
\vartheta(G)=\sum_{i=1}^n e_i \vartheta_{a_i}(\FD{G})= 0
$$
for every $G \in \mathcal{B}([0,1])$ then, using disjointness of the supports
it follows immediately that $e_1=e_2=\ldots=e_n=0$. \\
\FD{This transfers to the Pickands functions $A_i:=\Upsilon(\vartheta_{a_i})$: if $\sum_i e_iA_i\equiv0$, differentiating \eqref{eq:id_pick_meas} gives $\sum_i e_i\,(2\vartheta_{a_i}([0,t])-1)=0$ for a.e.\ $t$; letting $t\uparrow1$ yields $\sum_ie_i=0$, hence $\sum_i e_i\vartheta_{a_i}=0$ and $e_1=\dots=e_n=0$.} 
Having this, applying Theorem \ref{thm:keller} yields that 
$(\mathcal{A},\Vert \cdot \Vert_\infty)$ is homeomorphic to $(\mathcal{H},\rho)$ 
and the same holds true for $(\cdb^{ev},d_\infty)$. \\
It remains to prove that $\cdb^{ev}$ is a $Z$-set of $\cdb$. 
For every $a \in (0,1)$ define the mapping: $\kappa_a: \cdb \to \cdb$ by 
$\kappa_a(C):=a\,C + (1-a)W$. We want to show that $\kappa_a(C) \in \cdb \setminus \cdb^{ev}$ 
for every $a \in (0,1)$ and every $C \in \cdb$. For $G:=\{(x,1-x): x \in [0,1]\}$ 
we obviously have $\mu_{\kappa_a(C)}(G)>0$ for every $a \in (0,1)$. No EVC, however, assigns
positive mass to $G$ - in fact, according to \cite{DiTru,TSFS} the ($1$-to-$1$) Markov Kernel 
$K_{C_A}$ of every EVC $C_A$ with corresponding Pickands dependence measure $\vartheta$ 
fulfills the following equivalence for every $x,y \in (0,1)$:
\begin{align}
K_{C_A}(x,\{y\}) >0 \quad \Longleftrightarrow \quad t:=\tfrac{\ln(x)}{\ln(xy)} \text{ fulfills }
\vartheta(\{t\})>0
\end{align} 
In particular, $K_{C_A}(x,\{1-x\}) >0 $ if and only if 
$\vartheta(\{\frac{\ln(x)}{\ln(x(1-x))}\})>0$. Considering that, firstly, 
$x \mapsto \frac{\ln(x)}{\ln(x(1-x))}$ is a strictly decreasing homeomorphism of $[0,1]$, and that, 
secondly, $\vartheta$ can at most have countably many point masses, it follows that 
there are at most countably many $x \in (0,1)$ with $K_{C_A}(x,\{1-x\}) >0 $. 
Using disintegration of $G$ therefore yields
\begin{align*}
\mu_{C_A}(G)&= \int_{[0,1]} K_{C_A}(x,G_x) d\lambda(x) = \int_{[0,1]} K_{C_A}(x,\{1-x\}) d\lambda(x)=0.
\end{align*}  
In other words, $\kappa_a$ maps $\cdb$ into $\cdb \setminus \cdb^{ev}$ for every $a \in (0,1)$.
For sufficiently large $a$ the mapping $\kappa_a$ is arbitrarily close to the identity mapping on 
$\cdb$, so $\cdb^{ev}$ is a $Z$-set of $\cdb$ and the proof is complete. 
\end{proof}

Using Theorem \ref{thm:main.tool} (in the version sketched in 
Remark \ref{rem:alternative.main}) with $n=2,A_1=\cdb^{ev},A_2=\cdb$, we obtain the following result:
\begin{thm}\label{thm:main.evc}
There exists a homeomorphism $h : [0,1] \times \mathcal{H} \to \cdb$ with 
$h(\{0\} \times \mathcal{H})= \cdb^{ev}$. Moreover, $\FD{\cdb^{ev}}$ has the 
DC and the FP property and
is a retract of $\cdb$. 
\end{thm}

\subsection{Schur-concave and weakly Schur-concave copulas}

We quickly recall the notions of Schur-concavity and weak Schur-concavity for bivariate 
copulas (see \cite{DurantePapini2007,DuSe,Liu2024a}). A copula $C \in \cdb$ is called 
\emph{Schur-concave} (SC, for short) if 
$$
C(x,y) \leq C\big(\lambda x + (1-\lambda) y,\, (1-\lambda) x + \lambda y\big)
$$
holds for all $x,y,\lambda \in [0,1]$. Following \cite{DurantePapini2007}, 
$C \in \cdb$ is called \emph{weakly Schur-concave} (WSC, for short) if 
$$
C(x,y) \leq C\left(\tfrac{x+y}{2},\tfrac{x+y}{2}\right)=\delta_C\left(\tfrac{x+y}{2}\right)
$$
holds for all $x,y \in [0,1]$. We will let $\cdb^{sc}$ and $\cdb^{wsc}$ denote the family of all 
SC and WSC copulas, respectively, and write $\cdbs^{wsc}$ for the family of all exchangeable 
elements of $\cdb^{wsc}$. Choosing $\lambda=0$ and $\lambda=\frac12$ in the definition shows that 
every SC copula is exchangeable and WSC, so we have
$$
\cdb^{sc} \subseteq \cdbs^{wsc} \subseteq \cdb^{wsc} \subseteq \cdb .
$$
Since both properties are expressed by families of inequalities which are linear in $C$ and 
stable under pointwise limits, $\cdb^{sc}$, $\cdbs^{wsc}$ and $\cdb^{wsc}$ are convex and closed, 
hence compact, subsets of $(\cdb,d_\infty)$. In \cite{Liu2024a} it was shown that 
$(\cdb^{sc},d_\infty)$ is homeomorphic to the Hilbert cube $\mathcal{H}$. 
Contrary to Schur-concave copulas, however, weakly Schur-concave copulas are not necessarily 
exchangeable (see \cite{DurantePapini2007}); in fact, the maximal attainable asymmetry of WSC 
copulas w.r.t.\ $d_\infty$ is $\frac{1}{4}$, which is not too far away from the global bound 
$\frac13$ for all bivariate copulas established in \cite{KM,N}. As a consequence, 
\cite[Remark 5.4]{Liu2024a} is not correct and the chain of inclusions underlying 
\cite[Problem 5.1]{Liu2024a} does not hold. In the sequel we show that the natural corrected 
version of this problem, involving the chain displayed above, has a positive answer.

We start with a simple observation, which provides a large supply of (weakly) Schur-concave 
copulas.
\begin{lem}\label{lem:schur.mixture}
For every $D \in \cdb$ the copula $\frac12(M+D)$ is WSC. Moreover, $\frac12(M+D)$ is SC if and 
only if $D$ is exchangeable.
\end{lem}
\begin{proof}
Let $x\le y$, $h\in[0,\frac{y-x}{2}]$, and set $x_h:=x+h$, $y_h:=y-h$, so that $x_h \leq y_h$. 
Since $D$ is non-decreasing and $1$-Lipschitz in each argument we have 
$D(x_h,y_h)-D(x,y) \geq D(x,y_h)-D(x,y) \geq -h$, whereas $M(x_h,y_h)-M(x,y)=h$. Hence 
$h \mapsto \frac12(M+D)(x_h,y_h)$ is non-decreasing on $[0,\frac{y-x}{2}]$; for $h=\frac{y-x}{2}$ 
this yields weak Schur-concavity. If $D$ is exchangeable, so is $\frac12(M+D)$, and the 
monotonicity just shown is exactly Schur-concavity. Conversely, every SC copula is exchangeable, 
and $\frac12(M+D)$ is exchangeable if and only if $D$ is. 
\end{proof}

\begin{lem}\label{lem:schur.os}
Let $\theta \in (0,1)$ and $C_1,C_2 \in \cdb^{wsc}$. Then $\mathcal{OS}_\theta(C_1,C_2) \in \cdb^{wsc}$, 
and $\mathcal{OS}_\theta(C_1,C_2)$ is exchangeable whenever $C_1$ and $C_2$ are. 
Moreover, if $C_2$ is not SC (not exchangeable, not WSC, respectively), then 
$\mathcal{OS}_\theta(C_1,C_2)$ is not SC (not exchangeable, not WSC, respectively).
\end{lem}
\begin{proof}
Set $O:=\mathcal{OS}_\theta(C_1,C_2)$, let $x \leq y$ and $t:=\frac{x+y}{2}$. If $y \leq \theta$ 
or $x > \theta$, then $O(x,y) \le O(t,t)$ follows directly from the WSC property of $C_1$ or 
$C_2$, respectively, since the affine rescalings involved in \eqref{eq:os} preserve midpoints. 
If $x\le\theta<y$, then $O(x,y)=x$; for $t>\theta$ we get $O(t,t)\ge\theta\ge x$, and for 
$t \le \theta$ we get 
$O(t,t)=\theta\,\delta_{C_1}(\frac{t}{\theta})\ge \theta\,(\frac{2t}{\theta}-1)=x+y-\theta \geq x$. 
Hence $O \in \cdb^{wsc}$. Exchangeability of $O$ for exchangeable $C_1,C_2$ is obvious. 
Finally, on $[\theta,1]^2$ the copula $O$ is an affinely rescaled copy of $C_2$, and the 
rescaling $u\mapsto\theta+(1-\theta)u$ maps lines $\{u+v=\mathrm{const}\}$, midpoints and the 
reflection $(u,v)\mapsto(v,u)$ onto the corresponding objects in $[\theta,1]^2$. Therefore every 
violation of SC, exchangeability or WSC by $C_2$ yields a violation of the same property by $O$.
\end{proof}

The following example shows that $\cdb^{sc}$ is a proper subset of $\cdbs^{wsc}$.
\begin{exa}\label{ex:wsc.not.sc}
For $a \in [\frac{2}{3},1)$ set $c:=1-a \in (0,\frac13]$ and consider the $\lambda$-preserving 
transformation $T_a:[0,1]\to[0,1]$, given by
$$
T_a(x)=(x+a)\mathbf{1}_{[0,c)}(x) + x\,\mathbf{1}_{[c,a]}(x) + (x-a)\mathbf{1}_{(a,1]}(x).
$$
Since $T_a \circ T_a=\id$ holds $\lambda$-almost everywhere, the completely dependent copula 
$C_a$ assigning all mass to the graph $\Gamma(T_a)$ is exchangeable. 
Splitting the mass of $C_a$ according to the three segments of $\Gamma(T_a)$, for $x\le y$ we get
$$
C_a(x,y)=\big(\min\{x,y-a\}\big)^+ + \big(\min\{x,a\}-c\big)^+ + (x-a)^+,
\qquad 
\delta_{C_a}(t)=\big(\min\{t,a\}-c\big)^+ + 2\,(t-a)^+ .
$$
In particular, $C_a(\frac c2,1-\frac c2)=\frac c2>0=C_a(c,a)$; since $(c,a)$ lies on the 
segment joining $(\frac c2,1-\frac c2)$ and $(1-\frac c2,\frac c2)$, $C_a$ is not SC. \\
Nevertheless, $C_a$ is WSC. In fact, let $x \le y$ and $t:=\frac{x+y}{2}$. 
If $x \ge a$, then $C_a(x,y)=x+y-a-c=\delta_{C_a}(t)$. 
If $c\le x<a$ and $t\ge a$, then $C_a(x,y)\le (y-a)+(x-c)=\delta_{C_a}(t)$; if $c \leq x<a$ and 
$t<a$, then $\delta_{C_a}(t)=t-c$ and $C_a(x,y)\le (y-a)^+ + x-c \le t-c$, since $y-a<y-t$ 
whenever $y>a$. 
Finally, let $x<c$, so that $C_a(x,y)=(\min\{x,y-a\})^+$, which vanishes for $y \leq a$. 
If $y>a$ and $t \ge a$, then $\delta_{C_a}(t) \ge a-c \ge c > x$ because $a\ge\frac23$. 
If $y>a$ and $t<a$, then $\delta_{C_a}(t)\ge t-c$, and we have $t-c\ge x$ if $y\ge x+2c$ as well 
as $t-c\ge y-a$ if $y\le x+4a-2$; since $4a-2 \ge 2c$ (again because $a\ge\frac23$), one of these 
two cases always applies. Altogether $C_a(x,y)\le\delta_{C_a}(\frac{x+y}{2})$ for $x \leq y$, 
and by exchangeability for all $x,y\in[0,1]$. \\
In other words, for every $a \in [\frac{2}{3},1)$ the copula $C_a$ is an element of 
$\cdbs^{wsc}\setminus\cdb^{sc}$. 
Notice that the bound $a\geq\frac{2}{3}$ is sharp: for $a \in (\frac12,\frac23)$ and sufficiently 
small $x>0$ we have $C_a(x,x+a)=x>\big(x+\frac{3a}{2}-1\big)^+=\delta_{C_a}(x+\frac a2)$.
\end{exa}

\begin{lem}\label{lem:schur}
$\cdb^{sc}$, $\cdbs^{wsc}$ and $\cdb^{wsc}$ are compact, convex subsets of $(\cdb,d_\infty)$ and 
$\operatorname{ind}\cdb^{sc}=\infty$. Moreover, $\cdb^{sc}$ is a $Z$-set of $\cdbs^{wsc}$, 
$\cdbs^{wsc}$ is a $Z$-set of $\cdb^{wsc}$, and $\cdb^{wsc}$ is a $Z$-set of $\cdb$.
\end{lem}
\begin{proof}
Compactness and convexity have already been discussed above. According to 
Lemma \ref{lem:schur.mixture}, the mapping $D\mapsto\frac12(M+D)$ is an affine (hence continuous) 
injection of $\cdbs$ into $\cdb^{sc}$; by compactness it is a homeomorphism onto its image. 
Since $\operatorname{ind}\cdbs=\infty$ (see Lemma \ref{lem:mini.infty} with $d=2$, noting 
$\hat{\mathcal{C}}_2^{\Pi_1}=\cdbs$), Lemma \ref{lem:ind.mono} yields 
$\operatorname{ind}\cdb^{sc}=\infty$. \\
In order to prove the $Z$-set assertions we fix some $C^\ast \in \cdb$ and, for every 
$\theta \in (0,1)$, consider the continuous mapping $f_\theta: C \mapsto \mathcal{OS}_\theta(C,C^\ast)$. 
Using the Lipschitz continuity of copulas it is straightforward to verify that 
$\sup_{C \in \cdb} d_\infty(f_\theta(C),C) \leq 3(1-\theta)$, so $f_\theta$ is arbitrarily close 
to the identity for $\theta$ sufficiently close to $1$. \\
(i) Choosing $C^\ast := C_a$ from Example \ref{ex:wsc.not.sc}, Lemma \ref{lem:schur.os} shows 
that $f_\theta$ maps $\cdbs^{wsc}$ into $\cdbs^{wsc}\setminus\cdb^{sc}$, so $\cdb^{sc}$ is a 
$Z$-set of $\cdbs^{wsc}$. \\
(ii) Choosing $C^\ast := \frac12(M+D)$ with an arbitrary non-exchangeable copula $D$, 
Lemmas \ref{lem:schur.mixture} and \ref{lem:schur.os} show that $f_\theta$ maps $\cdb^{wsc}$ into 
$\cdb^{wsc}\setminus\cdbs^{wsc}$, so $\cdbs^{wsc}$ is a $Z$-set of $\cdb^{wsc}$. \\
(iii) Let $C^\ast$ be the completely dependent copula assigning all mass to the graph of 
$x\mapsto (x+\frac12)\,(\mathrm{mod}\, 1)$. Then 
$C^\ast(\frac15,\frac34)=\frac15>0=\delta_{C^\ast}(\frac{19}{40})$, so $C^\ast$ is not WSC, and by 
Lemma \ref{lem:schur.os} $f_\theta$ maps $\cdb$ into $\cdb\setminus\cdb^{wsc}$. Hence $\cdb^{wsc}$ 
is a $Z$-set of $\cdb$.
\end{proof}

Applying Theorem \ref{thm:main.tool} for the case $n=4$, $A_1=\cdb^{sc}$, $A_2=\cdbs^{wsc}$, 
$A_3=\cdb^{wsc}$, $A_4=\cdb$ and $N=C([0,1]^2)$ we obtain the following result:
\begin{thm}\label{thm:main.schur}
There exists a homeomorphism $h : [0,1]^3 \times \mathcal{H} \to \cdb$ with 
$$
h(\{0\}\times [0,1]^2 \times \mathcal{H})= \cdb^{wsc},\quad 
h(\{0\}^2\times [0,1] \times \mathcal{H})= \cdbs^{wsc},\quad
h(\{0\}^3 \times \mathcal{H}) = \cdb^{sc}.
$$
Moreover, each of the families $\cdb^{sc}$, $\cdbs^{wsc}$ and $\cdb^{wsc}$ has the DC and the FP 
property, $\cdb^{sc}$ is a retract of $\cdbs^{wsc}$, $\cdbs^{wsc}$ is a retract of $\cdb^{wsc}$, 
and $\cdb^{wsc}$ is a retract of $\cdb$.
\end{thm}
\begin{rem}
As a by-product, Lemma \ref{lem:schur} together with Theorem \ref{thm:keller} provides a short 
alternative proof of the main result of \cite{Liu2024a}, namely that $(\cdb^{sc},d_\infty)$ is 
homeomorphic to the Hilbert cube $\mathcal{H}$.
\end{rem}

\section{Properties of $\cd$ and $\qd$ and their exchangeable subclasses}

It has already been shown in \cite{LiuYang2023} that $\cd$ is homeomorphic to the Hilbert
cube $\mathcal{H}$. We here provide an alternative simple proof building upon 
Theorem \ref{thm:main.tool} and Lemma \ref{lem:dimension}. 
\FD{Throughout this section (as far as the families $\cdu,\cdus$ are concerned) we assume $d\ge3$: for $d=2$ we have $\mathcal{C}_2^{\Pi_1}=\mathcal{C}_2$.} 
We start, however, with a seemingly very restrictive class, the family $\cdu$ of all 
$d$-dimensional copulas whose $(d-1)$-dimensional marginals are all equal to the $(d-1)$-dimensional 
product copula $\Pi_{d-1}$. In other words: we study the family of distributions of 
$d$-dimensional random vectors $\mathbf{X}=(X_1,\ldots,X_d)$ of $\mathcal{U}(0,1)$-random variables
such that each $(d-1)$-\FD{tuple} is independent. $\cdus$ will denote the subset of all $C \in \cdu$ which
are exchangeable (symmetric).
In \cite{KBST} it was shown in terms of the metric $d_\infty$ 
(and another more restrictive metric $D_1$) that $\cdu$ is convex and compact and larger than one might expect - 
the following result points in this very direction too. 
\begin{lem}\label{lem:mini.infty}
The set $\cdus$ is closed, convex, and contains infinitely many linearly independent elements. 
In particular, $\operatorname{ind}\cdus=\infty$.  
\end{lem}
\begin{proof}
Obviously $\cdus$ is closed and convex, so it suffices to prove that $\cdus$ 
contains infinitely many linearly independent elements 
(linearly independent as elements of the Banach space $(C([0,1]^d),\Vert \cdot \Vert_\infty)$\FD{)}. 
Suppose that $a \in (0,1)$, let $X_1,\ldots,X_{d-1}$ be independent and uniformly distributed on 
$[0,1]$ and set 
$$
X_d:= 1- \left((a + X_1+X_2+\ldots+X_{d-1}) (\text{mod1}) \right).
$$  
Then $X_d$ is uniformly distributed on $[0,1]$, we have that $(a + \sum_{i=1}^d X_i) (\text{mod1})=0$, and that 
$$
X_j = 1- \left(\left( a + \sum_{i \neq j} X_i \right)(\text{mod1})\right)
$$
almost surely. 
Moreover, proceeding as in \cite{KBST} it is straightforward to verify that
the corresponding 
copula $C_a$ of $(X_1,\ldots,\FD{X_d})$ fulfills $C_a \in \cdus \subseteq \cdu$, 
and that $C_a$ is completely dependent and assigns all mass to the 
graph $\Gamma(T_a)$ of the function $T_a: [0,1]^{d-1} \rightarrow \FD{(0,1]}$, given by 
$T_a(x_1,\ldots,x_{d-1})= 1-(a + x_1+x_2+\ldots+x_{d-1})(\text{mod1})$. Figure \ref{fig:modulo}
depicts a sample of the copula $C_a \in \cdus$ for $d=3$ and $a=\frac{1}{4}$.\\
Suppose now that $a_1,\ldots,a_n$ are pairwise different elements in $(0,1)$. Then the copulas  
$C_{a_1},\ldots,C_{a_n}$ have pairwise disjoint supports 
$\operatorname{supp}(C_{a_1}),\ldots,\operatorname{supp}(C_{a_n})$  (consisting of segments of hyperplanes) \FD{and, in particular, $\mu_{C_{a_i}}(\Gamma(T_{a_j}))=\mathbf{1}_{\{i=j\}}$}. 
Assume that there are constants $e_1,\ldots,e_n \in \mathbb{R}$ such that we have 
$\sum_{i=1}^n e_i C_{a_i} = \mathbf{0} \in C([0,1]^d)$. Then the finite signed measure
$\vartheta$, given by 
$$
\vartheta(G):= \sum_{i=1}^n e_i \, \mu_{C_{a_i}}(G),\quad G \in \mathcal{B}([0,1]^d)
$$  
is the zero measure, i.e., $\vartheta(G)=0$ holds for every $G \in \mathcal{B}([0,1]^d)$. 
Considering $G=\Gamma(T_{a_1})$ and using disjointness of the supports it follows that
\begin{align*}
0=\vartheta(\Gamma(T_{a_1}))= \sum_{i=1}^n e_i \, \mu_{C_{a_i}}(\Gamma(T_{a_1})) = 
e_1 \mu_{C_{a_1}}(\Gamma(T_{a_1})) = e_1.
\end{align*}
Proceeding analogously eventually yields $e_1=e_2=\ldots=e_n=0$, implying that 
$C_{a_1},\ldots,C_{a_n}$ are linearly independent. Since $n$ was arbitrary it altogether follows that
$\cdus$ contains infinitely many linearly independent elements, so 
a direct application of Lemma \ref{lem:dimension} completes the proof. 
\end{proof}

\begin{figure}[htbp]
  \centering
  \includegraphics[width=0.5\textwidth]{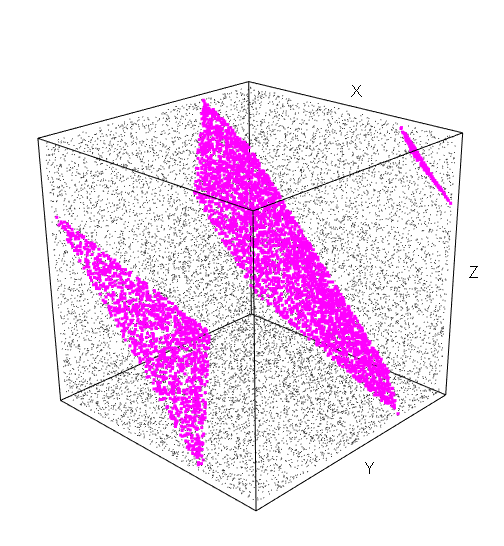}
  \caption{A sample of size $n=\FD{5{,}000}$ for the copula $C_a \in \hat{\mathcal{C}}_{3}^{\Pi_2}$ 
  with $a=\frac{1}{4}$ as considered in the proof of Lemma \ref{lem:mini.infty} (magenta), plus all bivariate marginal samples (gray).}
  \label{fig:modulo}
\end{figure}

\begin{thm}\FD{\label{thm:cdu}}
\FD{Let $d\ge3$.} $(\cdus,d_\infty)$ and $(\cdu,d_\infty)$ are homeomorphic to the Hilbert cube $(\mathcal{H},\rho)$. 
Moreover, $\cdus$ is a \FD{$Z$-set} of $\cdu$ and $\cdu$ is a $Z$-set \FD{of} $\cd$, and  
there exists some homeomorphism $h: [0,1]^2 \times \mathcal{H} \rightarrow \cd$ with 
$$
h(\{0\} \times [0,1] \times \mathcal{H})=\cdu, \quad h(\{0\} \times \{0\} \times \mathcal{H})=\cdus.
$$
Both, $\cdus$ and $\cdu$ have the DC and the FP property, $\cdus$ is a retract of $\cdu$ and 
$\cdu$ is a retract of $\cd$.
\end{thm} 
\begin{proof}
Using Lemma \ref{lem:mini.infty} and applying Theorem \ref{thm:keller} directly yields the 
first assertion. 
Considering Theorem \ref{thm:main.tool} it suffices to show that $\cdus$ is Z-set of $\cdu$ and 
that $\cdu$ is Z-set of $\cd$, which can be done as follows: \\
(i) We show that $\cdus$ is a $Z$-set of $\cdu$ and work with so-called generalized 
transformation matrices using the results from \cite{dstoch}. 
For every $\mathbf{i} \in  \{1,2\}^d $ and $l \in [d]$ define $\mathcal{O}_{l}: \{1,2\}^d \to \{0,\ldots,d\}$ by ($\mathcal{O}$ \FD{stands} for the number of ones)
$$
\mathcal{O}_{l}(\mathbf{i}) = \sum_{j=l}^d \mathbf{1}_{\{1\}}(i_j).
$$
For every $a \in (\frac{1}{2},1)$ define the probability distribution $\tau$ on $\{1,2\}^d$ by 
(we write $\tau(\mathbf{i})$ instead of $\tau(\{\mathbf{i}\})$)
\begin{equation}
\tau(\mathbf{i}) = a^{\mathcal{O}_{1}(\mathbf{i})}(1-a)^{d-\mathcal{O}_{1}(\mathbf{i})}\, + \, 
(-1)^{d-\mathcal{O}_{1}(\mathbf{i})}(1-a)^d. 
\end{equation}
Considering that for $a\in (\frac{1}{2},1)$ we have $1-a < \frac{1}{2} < a$ it is straightforward 
to verify that $\tau(\mathbf{i}) \in [0,1]$ for every $\mathbf{i} \in \{1,2\}^d$, that \FD{$\tau(\mathbf{i})=(1-a)^{d-k}\big(a^{k}+(-1)^{d-k}(1-a)^{k}\big)$ with $k=\mathcal{O}_1(\mathbf{i})$, so $\tau(\mathbf{i})=0$ if and only if $\mathbf{i}=(2,\ldots,2)$ and $d$ is odd; in particular $\tau((1,\ldots,1))=a^d+(1-a)^d>0$}, that  
$\tau$ is really a probability distribution and that 
$\tau$ is invariant to permutations. Moreover, a simple calculation shows that every 
univariate marginal distribution of $\tau$ assigns mass $a$ to $\{1\}$ and 
$(1-a)$ to $\{2\}$. For every $j \in [d]$ set $E^j_1=[0,a], E^j_2=[a,1]$ as well as 
$e^j_0=0,e^j_1=a,e^j_2=1$. 
Then for every $(i_2,\ldots,i_d) \in \{1,2\}^{d-1}$ we get
\begin{align*}
\tau(\{1,2\}\times \{i_2\}\times \cdots \times \{i_d\} ) &= 
\tau((1,i_2,\ldots,i_d)) + \tau((2,i_2,\ldots,i_d)) \\
&= a^{1+\mathcal{O}_2(\mathbf{i})}\,(1-a)^{d-1-\mathcal{O}_{2}(\mathbf{i})} + 
   (-1)^{d-1-\mathcal{O}_{2}(\mathbf{i})}(1-a)^d \\
   & \quad + a^{\mathcal{O}_2(\mathbf{i})}\,(1-a)^{d-\mathcal{O}_{2}(\mathbf{i})} + 
   (-1)^{d-\mathcal{O}_{2}(\mathbf{i})}(1-a)^d \\
   &= a^{\mathcal{O}_2(\mathbf{i})}\,(1-a)^{d-1-\mathcal{O}_{2}(\mathbf{i})} 
   = \prod_{j=2}^d \lambda(E^j_{i_j}),
\end{align*}
so $\tau$ fulfills the uniformity condition w.r.t. the first coordinate as stated in 
Definition 3.2 in \cite{dstoch}. Using permutation invariance of $\tau$ it follows that $\tau$ 
fulfills the uniformity condition w.r.t. every coordinate.  \\
For every $\mathbf{i} \in \{1,2\}^d$, the probability distribution $\tau$ induces 
an affine contraction $f_{\mathbf{i}} : [0,1]^d \to \bigtimes_{j=1}^d E^j_{i_j}$, given by
\[
f_{\mathbf{i}}(x_1, \ldots, x_d) :=
\begin{pmatrix}
e_{i_1-1}^1 \\
e_{i_2-1}^2 \\
\vdots \\
e_{i_d-1}^d
\end{pmatrix}
+
\begin{pmatrix}
(e_{i_1}^1 - e_{i_1-1}^1)\, x_1 \\
(e_{i_2}^2 - e_{i_2-1}^2)\, x_2 \\
\vdots \\
(e_{i_d}^d - e_{i_d-1}^d)\, x_d
\end{pmatrix}\FD{.}
\]
Letting $C_{\mathbf{i}}$ denote an arbitrary element of $\cdu$ for every 
$\mathbf{i} \in \{1,2\}^d$, it is straightforward to verify that $\FD{\nu_a}$, defined by 
(as before $\mu_{C_\mathbf{i}}^{f_{\mathbf{i}}}$ denotes the push-forward of $\mu_{C_\mathbf{i}}$ via $f_{\mathbf{i}}$)
$$
\nu_a(G)= \sum_{\mathbf{i} \in \{1,2\}^d: \tau(\mathbf{i})>0} \tau(\mathbf{i})\, 
\mu_{C_\mathbf{i}}^{f_{\mathbf{i}}}(G), \quad G \in \mathcal{B}([0,1]^d)
$$
is a $d$-stochastic measure. Moreover, according to \cite[Theorem 3.4]{dstoch}, the 
corresponding copula $C_{\nu_a}$ is an element of $\cdu$ since $\tau$ fulfills the uniformity condition w.r.t. \FD{every} coordinate. \\
Let $R \in \cdu \setminus  \cdus$ be arbitrary but fixed. 
For every $C \in \cdu$ set  
$$ 
C_{\mathbf{i}}:=
\begin{cases}
C & \text{if } \mathbf{i}=(1,1,\ldots,1), \\[3mm]
R & \text{otherwise.} 
\end{cases}
$$
Then writing $\varphi_{a,R}(C):=C_{\nu_a}$ defines a mapping $\varphi_{a,R}: \cdu \to \cdu$.  
Since $R$ is non\FD{-}exchangeable, the copula $\varphi_{a,R}(C)$ is not exchangeable for any 
$C \in \cdu$\FD{: choose a transposition $\kappa=(p\,q)$ under which $\mu_R$ is not invariant and $\mathbf{i}$ with $i_p=i_q=2$ and $\tau(\mathbf{i})>0$ (e.g.\ exactly one other coordinate equal to $1$); the box $\bigtimes_j E^j_{i_j}$ is invariant under $\kappa$ and there $\varphi_{a,R}(C)$ is a rescaled copy of $R$}.   
For every fixed $a \in (\frac{1}{2},1)$ this defines a mapping $\varphi_{a,R}$ from 
$\cdu$ to $\cdu \setminus \cdus$ and this mapping is easily seen to be continuous. \\
Finally, for every $\varepsilon>0$ we can obviously choose $a \in (\frac{1}{2},1)$ sufficiently close to 
$1$ so that 
$$
\sup_{C \in \cdu} d_\infty \left(\varphi_{a,R}(C),C \right) < \varepsilon
$$  
holds. This shows that $\cdus$ is a $Z$-set of $\cdu$. \\
(ii) Fix some $R \in \cd$ and let $\varepsilon>0$ be arbitrary but fixed. 
Choosing $a \in (0,1)$ sufficiently close to $1$ we obtain   
$$
\sup_{C \in \cd} d_\infty \left(\mathcal{OS}_a(C,R),C \right) < \varepsilon,
$$ 
which completes the proof since obviously $\mathcal{OS}_a(C,R) \in \cd \setminus \cdu$
holds for every copula $C \in \cd$.
\end{proof} 
We now focus on the full families $\cd$ and $\qd$. 
\begin{lem}\label{lem:cdqd}
$\cd$ is a $Z$-set of $\qd$, $\cds$ is a $Z$-set of $\cd$, and 
$\qds$ is a $Z$-set of $\qd$. 
\end{lem}
\begin{proof}
It has already \FD{been} mentioned that $\cd$ is a closed convex subset of $\qd$. Let $\varepsilon>0$
be arbitrary but fixed and consider some proper quasi-copula $R \in \qd\setminus \cd$. 
Then for every $a \in (0,1)$ the ordinal sum $\mathcal{OS}_a(Q,R)$ is a proper quasi-copula, i.e., 
$\mathcal{OS}_a(Q,R) \in \qd \setminus \cd$. 
Choosing $a \in (0,1)$ sufficiently close to $1$ we obtain   
$$
\sup_{Q \in \qd} d_\infty \left(\mathcal{OS}_a(Q,R),Q \right) < \varepsilon,
$$ 
which proves the first assertion. The other two statements can be proved in the same manner:  
In fact, letting $R \in \cd$ denote an arbitrary non-exchangeable copula, we have that 
$\mathcal{OS}_a(C,R) \in \cd$ is non-exchangeable for every $C \in \cd$; moreover,  
$\mathcal{OS}_a(Q,R) \in \qd \setminus \qds$ for every $Q \in \qd$. 
The results now follow from the fact that 
both, $\cds$ and $\qds$ are compact and hence closed.
\end{proof}

\FD{
\begin{rem}
Notice that, using an exchangeable $R$ one also gets that $\cdus$ is a $Z$-set of $\cds$ ($d\ge3$) and that $\cds$ is a $Z$-set of $\qds$ (exchangeable proper $R$), which would give the chain $\cdus\subset\cds\subset\qds\subset\qd$ and a homeomorphism $[0,1]^3\times\mathcal{H}\to\qd$.
\end{rem}}

\begin{thm}\label{thm:cdqd}
There exists a homeomorphism $h : [0,1] \times \mathcal{H} \to \qd$ with 
$h(\{0\} \times \mathcal{H})= \cd$. Both $\cd$ and $\qd$ have the DC and the FP property and 
$\cd$ is a retract of $\qd$. \\
Moreover, there exists a homeomorphism $g : [0,1] \times \mathcal{H} \to \qd$ with 
$g(\{0\} \times \mathcal{H})= \qds$. The family $\qds$ has the DC and the FP property and 
is a retract of $\qd$. \FD{Analogously for $\cds$ and $\cd$.} 
\end{thm}
\begin{proof}
According to Lemma \ref{lem:cdqd}, $\cd$ is a $Z$-set of $\qd$. Moreover, considering 
$\cdus \subset \cd \subset \qd$ and using monotonicity of the small inductive dimension, 
Lemma \ref{lem:mini.infty} yields $\operatorname{ind}\cd= \infty$. Therefore, 
considering $n=2, A_1=\cd, A_2=\qd$ and $N=C([0,1]^d)$ the stated result now directly follows from 
Theorem \ref{thm:main.tool}. The second assertion follows in the same manner.
\end{proof}
Since $[0,1] \times \mathcal{H}, \{0\} \times \mathcal{H}$ and $\mathcal{H}^k$ (with 
arbitrary $k \in \mathbb{N}$) are homeomorphic to $\mathcal{H}$, Theorem \ref{thm:cdqd} 
has the following immediate corollary ($\rho_k$ denoting the standard maximum extension of 
$\rho$ to $\mathcal{H}^k$ and $d_{\infty,k}$ the maximum extension of $d_\infty$ from 
$\qd$ to $\qd^k$): 
\begin{cor}
For every $k \in \mathbb{N}$ all the following metric spaces are homeomorphic: 
\begin{itemize}
\item $(\mathcal{H},\rho),\,(\mathcal{H}^k,\rho_k)$,
\item $(\cd^k,d_{\infty,k}),(\cd,d_{\infty}),(\cds,d_{\infty}),(\cdus,d_{\infty}), 
(\cdu,d_{\infty})$,
\item $(\qd^k,d_{\infty,k}),(\qd,d_{\infty}),(\qds,d_{\infty})$. 
\end{itemize}
\end{cor}

\FD{\section{Some families of multivariate copulas}

Analogously to the previous section about bivariate classes, here we investigate two relevant subclasses of copulas in dimension $d\ge 2$.
}
\subsection{Multivariate spatially homogeneous copulas}

We quickly recall the definition of multivariate spatially homogeneous copulas as 
motivated by \cite{Brown1966} in the context of Markov operators and studied in the copula
setting in \cite{DuranteFernandezTrutschnig2020}. 
For every $a \in \FD{[0,1)}$ we will let $r_a: [0,1] \to [0,1)$ denote the mapping (rotation), given by
$r_a(x)= (x+a)\text{mod1}$. It is straightforward to verify that each $r_a$ is 
$\lambda$-preserving, i.e., that $\lambda^{r_a}=\lambda$ holds. 
For every $a \in [0,1)$ define the multivariate version $R_a: [0,1]^{d-1} \to \FD{[0,1)^{d-1}}$ by
$$
R_a(y_2,\ldots,y_d)=\left(r_a(y_2),\ldots,r_a(y_d) \right)
$$
Following \cite{DuranteFernandezTrutschnig2020} we will call a copula $C \in \cd$ 
spatially homogeneous if and only if there exists some $\vartheta \in \FD{\mathcal{P}([0,1]^{d-1})}$
such that 
\begin{align}
K(x,E) := \vartheta^{R_x}(E)= \vartheta\left(R_x^{-1}(E) \right), \quad x \in [0,1],\, E \in \mathcal{B}([0,1]^{d-1})
\end{align}
is a $1$-to-$(d-1)$-Markov kernel of $C$. Indicating the dependence on $\vartheta$ we will 
let $C_\vartheta \in \cd$ denote \FD{the} spatially homogeneous copula induced by $\vartheta$ and 
write $\cd^{sh}$ for the family of all $d$-dimensional spatially homogeneous copulas. 
It is straightforward to verify that considering $\vartheta=\lambda_{d-1}$ yields $C_\vartheta=\Pi_d$
and $\vartheta=\FD{\epsilon_{\mathbf{0}}}$ yields $C_\vartheta=M$.
Obviously the family $\cd^{sh}$ is convex - in order to show compactness we will use the 
subsequent lemma. 
\begin{lem}\label{lem:weak.to.sh}
Suppose that $(\vartheta_n)_{n \in \mathbb{N}}$ is a sequence of probability measures in 
$\mathcal{P}([0,1]^{d-1})$ that converges weakly to $\vartheta \in \mathcal{P}([0,1]^{d-1})$. 
Then the corresponding induced spatially homogeneous copulas $C_{\vartheta_n}$ converge to 
$C_{\vartheta}$ with respect to $D_1$ for $n \to \infty$. 
\end{lem} 
\begin{proof}
We want to apply the Continuous Mapping theorem and proceed as follows: 
\FD{For every $i\in\{2,\ldots,d\}$ the one-dimensional marginal $\vartheta^{\pi_i}$ of $\vartheta$ has at most countably many atoms.} For every $x \in (0,1)$ the set $\operatorname{Disc}(R_x)$ of 
discontinuity points of $R_x$ obviously is of the form 
$$
\operatorname{Disc}(R_x) = \bigcup_{i=2}^d \left\{(x_2,\ldots,x_d) \in [0,1]^{d-1}: \, x_i=1-x \right\}.
$$ 
As a direct consequence \FD{$\vartheta(\operatorname{Disc}(R_x))\le\sum_{i=2}^d\vartheta^{\pi_i}(\{1-x\})=0$} \FD{holds} for all but at most 
countably many $x \in (0,1)$, and hence for $\lambda$-almost every $x \in (0,1)$.
Using Continuous Mapping theorem it therefore follows that for $\lambda$-almost every
$x \in (0,1)$ we have that $(\vartheta_n^{R_x})_{n \in \mathbb{N}}$ converges weakly to 
$\vartheta^{R_x}$. In other words: for $\lambda$-almost every $x \in (0,1)$ we have 
that $(K_{C_{\vartheta_n}}(x,\cdot))_{n \in \mathbb{N}}$ converges weakly to 
$K_{C_\vartheta}(x,\cdot)$. This shows $D_1$-convergence and completes the proof. 
\end{proof}
\begin{lem}\label{lem:main.cdsh}
$\cd^{sh}$ is a compact, convex subset of $(\cd,d_\infty)$ which contains infinitely many 
linearly independent elements. Moreover, $\cd^{sh}$ is a $Z$-set of $\cd$.
\end{lem}
\begin{proof}
We show that every sequence in $(\cd^{sh},d_\infty)$ contains a convergent subsequence with 
limit in $\cd^{sh}$. For $C_1,C_2,\ldots \in \cd^{sh}$ there exist probability measures
$\vartheta_1,\vartheta_2\FD{,\ldots} \in \mathcal{P}([0,1]^{d-1})$ with $C_{\vartheta_i}=C_i$ for every 
$i \in \mathbb{N}$. Letting $m_W$ denote the Wasserstein (a.k.a. Hutchinson or 
Kantorovich-Rubinstein) metric on $\mathcal{P}([0,1]^{d-1})$, it is 
well-known that $m_W$ 
is a metrization of weak convergence and that $(\mathcal{P}([0,1]^{d-1}),m_W)$ is compact. 
As a direct consequence, $(\vartheta_n)_{n \in \mathbb{N}}$ has a weakly convergent subsequence 
$(\vartheta_{n_j})_{j \in \mathbb{N}}$ with limit $\vartheta \in \mathcal{P}([0,1]^{d-1})$.
Applying Lemma \ref{lem:weak.to.sh} directly yields 
$\lim_{\FD{j} \rightarrow \infty} D_1(C_{\vartheta_{n_j}},C_\vartheta)=0$, which implies 
$\lim_{\FD{j} \rightarrow \infty} d_\infty(C_{\vartheta_{n_j}},C_\vartheta)=0$. We have shown that every
sequence has a convergent subsequence, so $(\cd^{sh},d_\infty)$ is compact. \\
For every $\mathbf{b} \in [0,1)^{d-1}$ consider $C_{\delta_{\mathbf{b}}}$. Then 
for $\mathbf{a} \neq \mathbf{b}$ the corresponding $d$-stochastic measures 
$\mu_{\mathbf{a}}:=\mu_{C_{\delta_{\mathbf{a}}}}$ and 
$\mu_{\mathbf{b}}:=\mu_{C_{\delta_{\mathbf{b}}}}$ have disjoint support. 
Suppose now that 
$\mathbf{a}^1=(a^1_{\FD{2}},\ldots,a^1_d),\ldots,\mathbf{a}^n=(a^n_{\FD{2}},\ldots,a^n_d)$ are pairwise different elements in $[0,1)^{d-1}$ 
and assume that there are constants $e_1,\ldots,e_n \in \mathbb{R}$ such that we have 
$\sum_{i=1}^n e_i C_{\delta_{\mathbf{a}^i}} = \mathbf{0} \in C([0,1]^d)$. Then the 
finite signed measure $\nu$, given by 
$$
\nu(G):= \sum_{i=1}^n e_i \, \mu_{C_{\delta_{\mathbf{a}^i}}}(G),\quad G \in \mathcal{B}([0,1]^d)
$$  
is the zero measure, i.e., $\nu(G)=0$ holds for every $G \in \mathcal{B}([0,1]^d)$. 
Considering $G=\{(x,\FD{r_x}(a^1_2),\ldots,\FD{r_x}(a^1_d)): \, x \in [0,1]\}$ and using disjointness of the supports it follows that
\begin{align*}
0=\nu(G)= \sum_{i=1}^n e_i \, \mu_{\mathbf{a}^i}(G) = 
e_1 \mu_{\mathbf{a}^1}\FD{(G)} = e_1.
\end{align*}
Proceeding analogously eventually yields $e_1=e_2=\ldots=e_n=0$, implying that 
$C_{\delta_{\mathbf{a}^{\FD{1}}}},\ldots,C_{\delta_{\mathbf{a}^{\FD{n}}}}$ are linearly independent. 
Since $n$ was arbitrary, this shows that $\cd^{sh}$ contains infinitely many 
linearly independent elements and it remains to prove the last assertion.\\
We again work with ordinal sums and for every $a \in (0,1)$ and $C \in \cd$ consider 
$\mathcal{OS}_a(C,\Pi)$. Then $\mathcal{OS}_a(C,\Pi) \in \cd \setminus \cd^{sh}$ 
for every $C \in \cd$\FD{: for $\lambda$-a.e.\ $x\in(a,1)$ its kernel $K(x,\cdot)$ is the uniform distribution on $[a,1]^{d-1}$; if it were of the form $\vartheta^{R_x}$, then $\vartheta$ would be uniform on $R_x^{-1}([a,1]^{d-1})$, a set depending on $x$ --- a contradiction}.  
For arbitrary $\varepsilon > 0$, choosing $a \in (0,1)$ sufficiently close to $1$ we obtain   
$$
\sup_{C \in \cd} d_\infty \left(\mathcal{OS}_a(C,\Pi),C \right) < \varepsilon,
$$ 
which shows that $\cd^{sh}$ is a $Z$-set of $\cd$. 
\end{proof}

\begin{figure}[htbp]
  \centering
  \includegraphics[width=0.5\textwidth]{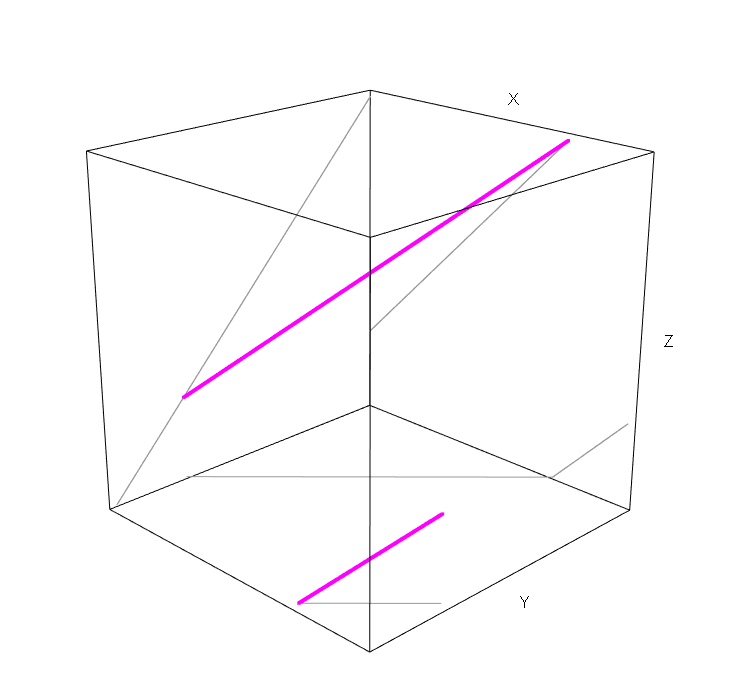}
  \caption{A sample of size $n=\FD{5{,}000}$ for the copula $C_{\delta_{\mathbf{a}}} \in \FD{\mathcal{C}_{3}^{sh}}$ 
  with $\mathbf{a}=(\frac{1}{4},\frac{1}{4})$ as considered in the proof of Lemma \ref{lem:main.cdsh} (magenta), plus all bivariate marginal samples (gray).}
  \label{fig:modulo2}
\end{figure}

Using Lemma \ref{lem:main.cdsh}, setting $n=2, A_1=\cd^{sh}, A_2=\cd$ and $N=C([0,1]^d)$, 
the following result directly follows from Theorem \ref{thm:main.tool}:
\begin{thm}
There exists a homeomorphism $h : [0,1] \times \mathcal{H} \to \cd$ with 
$h(\{0\} \times \mathcal{H})= \cd^{sh}$. Moreover, $\FD{\cd^{sh}}$ has the DC and the FP property and
 is a retract of $\cd$.
\end{thm}

\subsection{Multivariate $\alpha$-migrative copulas}

Following \cite{migrative,MesiarBustinceFernandez2010}, for given $\alpha \in (0,1)$ a 
copula $C \in \cd$ is called $\alpha$-migrative, if for all $i,j \in [d]$ and all 
$\mathbf{x}=(x_1,\ldots,x_d) \in [0,1]^d$ the following identity holds:
\begin{align}
C(x_1,\ldots,x_{i-1},\alpha x_i, x_{i+1},\ldots,x_d) = C(x_1,\ldots,x_{j-1},\alpha x_j, x_{j+1},\ldots,x_d)
\end{align}
We will let $\cd^{\alpha m}$ denote the family of all $d$-dimensional $\alpha$-migrative copulas. 
Obviously the product copula $\FD{\Pi_d}$ is $\alpha$-migrative for every $\alpha$, whereas for instance 
$M$ is easily verified to not be $\alpha$-migrative for any $\alpha \in (0,1)$.  
It is straightforward to verify that $\cd^{\alpha m}$ is a closed and convex subset of $\cd$. 
In this section we show that $\cd^{\alpha m}$ is a $Z$-set of $\cd$, that $\cd^{\alpha m}$ 
is homeomorphic to the Hilbert cube $\mathcal{H}$ and that there exists a homeomorphism 
$h: [0,1] \times \mathcal{H} \to \cd$ mapping $\{0\}\times \mathcal{H}$ to $\cd^{\alpha m}$. 
Doing so we first recall the handy characterization of $\cd^{\alpha m}$ going back to  
\cite{migrative}, formulate it in terms of the corresponding $d$-stochastic measures, and 
start with two notations: For every $i \in \{0,1,\ldots\}=\mathbb{N}_0$ set
$$
\Delta_i:=\left\{(k_1,\ldots,k_d) \in \mathbb{N}_0^d: \, \sum_{j=1}^d k_j=i\right\}
$$ 
and for every $\mathbf{k} \in \Delta_i$ define the affine (bijective) transformation 
$S_{\mathbf{k},i}: [0,1]^d \to \bigtimes_{j=1}^d [\alpha^{k_j+1},\alpha^{k_j}]$ 
coordinate-wise by
$$
\pi_j \circ S_{\mathbf{k},i}(x_1,\ldots,x_d)=\left(\alpha^{k_j+1} + x_j(\alpha^{k_j}-\alpha^{k_j+1}) \right),\quad j \in [d].
$$
According to \cite[Proposition 4.3]{migrative} the following two assertions are equivalent for
$C \in \cd$:
\begin{itemize}
\item $C \in \cd^{\alpha m}$
\item There are copulas $C_0,C_1,C_2,\ldots \in \cd$ such that the corresponding 
$d$-stochastic measures fulfill
\begin{align}\label{eq:formula.migrative}
\mu_C=(1-\alpha)^d \sum_{i=0}^\infty \alpha^i \sum_{\mathbf{k} \in \Delta_i} \mu_{C_i}^{S_{\mathbf{k},i}}.
\end{align}
\end{itemize}
In other words: considering the partition of $[0,1]^d$ into rectangles induced 
by $\{\alpha^l: l \in \mathbb{N}_0\}$ in each coordinate, a copula $C \in \cd$ is $\alpha$-migrative, 
if and only if $\mu_C$ is a patchwork-like construction consisting of affinely rescaled copies  
of $\mu_{C_i}$ in every rectangle of the form $\bigtimes_{j=1}^d [\alpha^{k_j+1},\alpha^{k_j}]$ with 
$\mathbf{k} \in \Delta_i$. In particular, for $i=0$ we have $\Delta_i=\{\mathbf{0}\}$, 
so for every Borel set $G \subseteq [\alpha,1]^d$ the measure $\mu_C$ fulfills 
$\mu_C(G)=(1-\alpha)^d \mu_{C_0}^{S_{\mathbf{0},0}}(G)$, implying that 
$C_0$ can be recovered from $C$ and every $\alpha$-migrative copula $C$ fulfills 
\begin{equation}\label{eq:upper.alpha.mig}
\mu_C\left([\alpha,1]^d\right)=(1-\alpha)^d\mu_{C_0}^{S_{\mathbf{0},0}}\left([\alpha,1]^d\right)=
(1-\alpha)^d\mu_{C_0}\left([0,1]^d\right)=(1-\alpha)^d
\end{equation}
Defining 
$\varphi: \cd \to \cd^{\alpha m}$ in terms of the corresponding $d$-stochastic measures by 
(all $C_i$ are equal to $C$) 
\begin{align}\label{eq:formula.migrative2}
\mu_{\varphi(C)}=(1-\alpha)^d \sum_{i=0}^\infty \alpha^i \sum_{\mathbf{k} \in \Delta_i} \mu_{C}^{S_{\mathbf{k},i}},
\end{align}  
it therefore follows that $\varphi$ is injective. Moreover, letting $C_b \in \cdus$ denote the copula
used in the proof of Lemma \ref{lem:mini.infty} it follows that for arbitrary 
$b_1,b_2 \in (0,1)$ we have that $d$-stochastic measures corresponding to 
the $\alpha$-migrative copulas $\varphi(C_{b_1}),\varphi(C_{b_2})$ 
have disjoint support \FD{(more precisely, they are mutually singular)}. Proceeding as in the proof of Lemma \ref{lem:mini.infty} and 
Lemma \ref{lem:main.cdsh} therefore yields that $\cd^{\alpha m}$ contains infinitely many 
linearly independent elements. Figure \ref{fig:migrative} depicts a sample of 
$\varphi(C_{\frac{1}{4}})$ for $\alpha=\frac{1}{2}$. \\
For every $a \in \FD{(\alpha,1)}$ and $C \in \cd$ consider 
$\mathcal{OS}_a(C,M) \in \cd$. Then it is straightforward to verify that 
$\mathcal{OS}_a(C,M) \in \cd \setminus \cd^{\alpha m}$ for every $C \in \cd$. \FD{In fact, $\mu_{\mathcal{OS}_a(C,M)}([a,1]^d)=1-a$, whereas for $D\in\cd^{\alpha m}$, by \eqref{eq:formula.migrative} with $b:=\frac{a-\alpha}{1-\alpha}$,
$$\mu_D([a,1]^d)=(1-\alpha)^d\mu_{D_0}([b,1]^d)\le(1-\alpha)^d(1-b)=(1-\alpha)^{d-1}(1-a)<1-a.$$
(The same argument works with $M$ replaced by any copula.)} \FD{Since for} 
arbitrary $\varepsilon > 0$, choosing $a \in (0,1)$ sufficiently close to $1$ we obtain   
$$
\sup_{C \in \cd} d_\infty \left(\mathcal{OS}_a(C,M),C \right) < \varepsilon,
$$ 
$\cd^{\alpha m}$ is a $Z$-set of $\cd$. Summing up, we have shown the following result:
\begin{lem}\label{lem:main.mig}
Let $\alpha \in (0,1)$ be arbitrary but fixed. Then $\cd^{\alpha m}$ is a compact, convex subset 
of $(\cd,d_\infty)$ which contains infinitely many 
linearly independent elements. Moreover, $\cd^{\alpha m}$ is a $Z$-set of $\cd$.
\end{lem}

Using Lemma \ref{lem:main.mig}, setting $n=2, A_1=\cd^{\alpha m}, A_2=\cd$ and $N=C([0,1]^d)$, 
the following result directly follows from Theorem \ref{thm:main.tool}:
\begin{thm}
Let $\alpha \in (0,1)$ be arbitrary but fixed. Then 
\FD{there} exists a homeomorphism $h : [0,1] \times \mathcal{H} \to \cd$ with 
$h(\{0\} \times \mathcal{H})= \cd^{\alpha m}$. Moreover, $\cd^{\alpha m}$ has the DC and the FP property and
 is a retract of $\cd$.
\end{thm}

\begin{figure}[h!]
  \centering
  \includegraphics[width=0.8\textwidth]{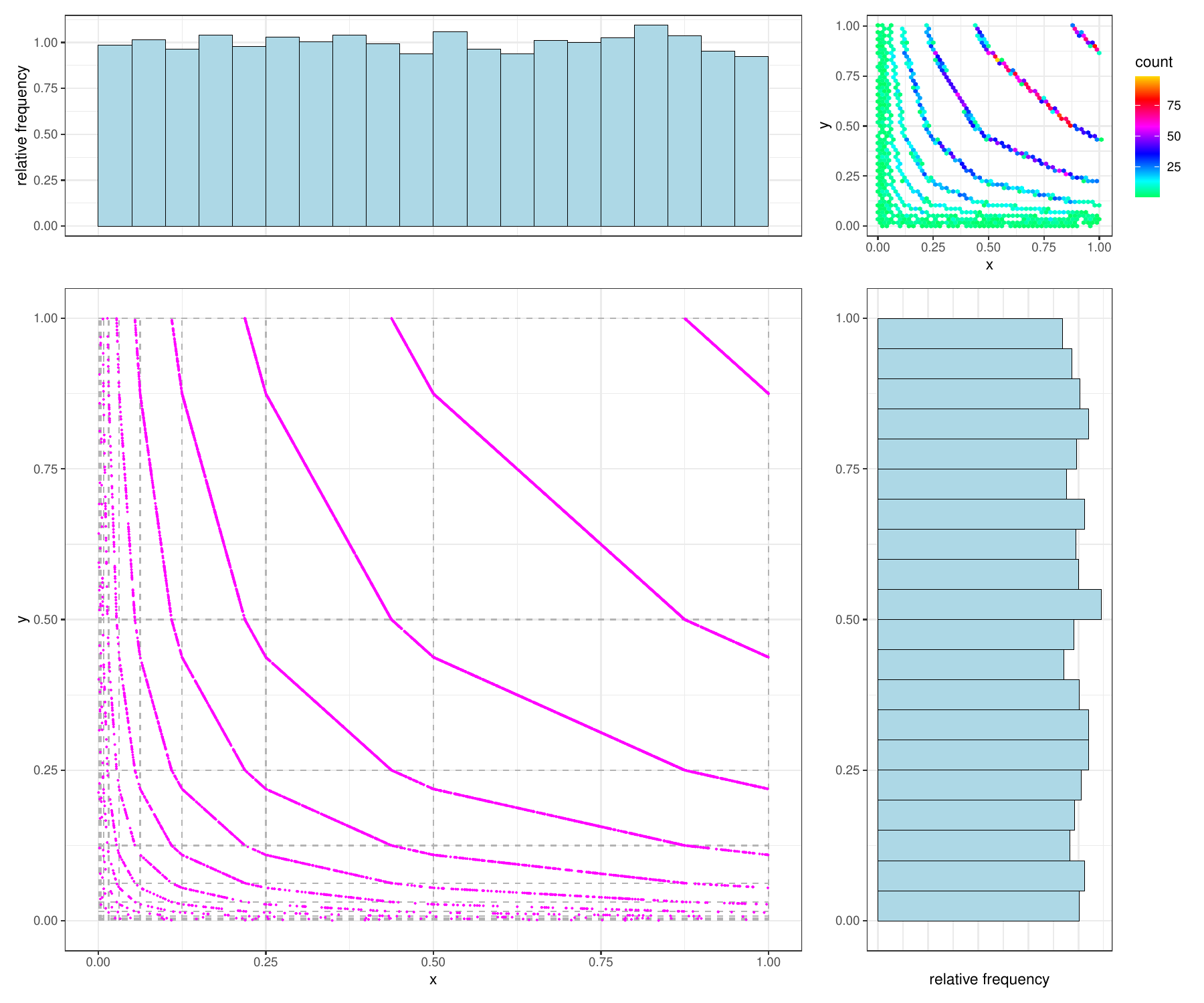}
  \caption{A sample of size $n=\FD{10{,}000}$ for the $\alpha$-migrative 
  copula $\varphi(C_{b})$ (upper left panel, magenta) with $C_b \in \FD{\hat{\mathcal{C}}_3^{\Pi_2}}$, 
  $b=\frac{1}{4}$, and $\alpha=\frac{1}{2}$, used for showing 
  that $\cd^{\alpha m}$ contains infinitely \FD{many linearly independent elements}; 
  bivariate marginal samples (blue) and hexagonal heatmap of the sample (upper right panel).}
  \label{fig:migrative}
\end{figure}

\section{Conclusions}
In this paper we have studied numerous wekk-known families of bivariate and multivariate copulas and 
quasi-copulas from the viewpoint of infinite-dimensional topology. Building upon the 
characterizations of the Hilbert cube due to Keller and Toru\'nczyk and on the $Z$-set 
unknotting theorem, Theorem~\ref{thm:main.tool} provides a simple and flexible tool which 
we applied to families defined via diagonal sections, lower semilinear, biconic, diagonal, 
Bertino and extreme-value copulas, (weakly) Schur-concave copulas, copulas with independent 
$(d-1)$-dimensional marginals, exchangeable, spatially homogeneous and $\alpha$-migrative copulas, 
as well as to the full families $\cd$ and $\qd$. In all cases the subfamily under consideration 
turned out to be homeomorphic to $\mathcal{H}$ and, at the same time, topologically negligible 
in the ambient family in the sense of $Z$-sets. In other words, from the topological point of 
view all these families are as rich as the whole class of copulas, whereas their position inside 
the ambient family is that of a face $\{0\}\times\mathcal{H}$ of $[0,1]\times\mathcal{H}$.

A direct consequence of the obtained homeomorphisms is worth mentioning: whenever 
$h:[0,1]\times\mathcal{H}\to\Omega$ is a homeomorphism with $h(\{0\}\times\mathcal{H})=A$, the 
complement $\Omega\setminus A$ is homeomorphic to $(0,1]\times\mathcal{H}$. For instance, the 
family $\qd\setminus\cd$ of all proper $d$-dimensional quasi-copulas and the family of all 
non-exchangeable copulas are homeomorphic to $(0,1]\times\mathcal{H}$, i.e., to the Hilbert cube 
with one point removed.

\section*{Acknowledgements}
\noindent The third author gratefully acknowledges the support of the WISS project 
‘IDA Lab Salzburg’ (20102/F2300464-KZP; 20204-WISS/225/348/3-2023).


\end{document}